\documentclass[12pt]{amsart}
\usepackage{etoolbox}

\makeatletter
\def\@secnumfont{\mdseries}
\def\section{\@startsection{section}{1}%
 \z@{.7\linespacing\@plus\linespacing}{.5\linespacing}%
 {\normalfont\scshape\centering}}
\def\subsection{\@startsection{subsection}{2}%
 \z@{.5\linespacing\@plus.7\linespacing}{-.5em}%
 {\normalfont\bfseries}}

\patchcmd{\@thm}{\let\thm@indent\indent}{\let\thm@indent\noindent}{}{}
\patchcmd{\@thm}{\thm@headfont{\scshape}}{\thm@headfont{\bfseries}}{}{}

\makeatother

\DeclareFontFamily{U}{matha}{\hyphenchar\font45}
\DeclareFontShape{U}{matha}{m}{n}{
      <5> <6> <7> <8> <9> <10> gen * matha
      <10.95> matha10 <12> <14.4> <17.28> <20.74> <24.88> matha12
      }{}
\DeclareSymbolFont{matha}{U}{matha}{m}{n}
\DeclareFontSubstitution{U}{matha}{m}{n}
\DeclareMathSymbol{\acap}{2}{matha}{"58}
\DeclareMathSymbol{\acup}{2}{matha}{"59}

\usepackage[dvips]{graphicx}
\usepackage{hyperref}
\usepackage{color,mathdots,accents,extarrows,bigints,enumitem,stmaryrd}
\usepackage[all]{xypic}
\usepackage[ansinew]{inputenc}
\usepackage{amsmath,amsfonts,amssymb}
\usepackage[all,graph]{xy}
\usepackage[shortcuts]{extdash}

\usepackage[capitalize]{cleveref}

\usepackage{tikz-cd}

\usepackage{comment}

\newtheorem*{d2}{D2 problem}

\newtheorem{theorem}{Theorem}[section]
\newtheorem{lemma}[theorem]{Lemma}

\newtheorem{corollary}[theorem]{Corollary}

\newtheorem{proposition}[theorem]{Proposition}
\newtheorem*{proposition*}{Proposition}
\newtheorem*{theorem*}{Theorem}
\newtheorem{question}[theorem]{Question}

\theoremstyle{remark}
\newtheorem*{remark}{Remark}
\newtheorem*{convention}{Convention}
\newtheorem*{definition}{Definition}
\newtheorem*{definition-notation}{Definition and notation}
\newtheorem*{notation}{Notation}

\makeatletter

\newcommand*{\textlabel}[2]{%
 \edef\@currentlabel{#1}
 \phantomsection
 #1\label{#2}
}
\makeatother

\newcommand{\ttmoplus}[2]{\mbox{\footnotesize{$\textstyle\bigoplus\limits_{#1}^{#2}$}}}

\def\lk{\op{lk}}
\def\Q{\mathbb{Q}} 
 
\def\Z{\mathbb{Z}} 
\def\R{\mathbb{R}} 
\def\C{\mathbb{C}}
\def\O{\mathbb{O}}
\def\N{\mathbb{N}} 
 
\def\Hr{\mathbb{H}}

\def\rp{\R\textup{P}}
\def\cp{\C\textup{P}}
\def\OP{\O\textup{P}}
\def\HP{\mathbb{H}\textup{P}}

\def\bnm{\begin{enumerate}} 
\def\enm{\end{enumerate}}

\def\sc{\op{sc}}
\def\high{\op{hc}}

\newcommand{\ttmfrac}[2]{\mbox{$\textstyle\frac{#1}{#2}$}}

\DeclareMathAlphabet{\mathbf}{OML}{cmm}{b}{it}

\def\bn{\begin{enumerate}} 
\def\en{\end{enumerate}}
\def\ba{\begin{array}} 
\def\ea{\end{array}} 
\def\bp{\begin{pmatrix}}
\def\ep{\end{pmatrix}}

\def\bn{\begin{enumerate}} 
\def\en{\end{enumerate}}
\def\ba{\begin{array}} 
\def\ea{\end{array}} 
\def\bp{\begin{pmatrix}}
\def\ep{\end{pmatrix}}

\def\op{\operatorname}

\def\PP{\mathcal{P}}

\def\MM{\mathcal{M}}

\def\ub{\underbrace}
\def\us{\underset}

\def\ol{\overline}

\def\rp{\R\textup{P}}

\def\Def{\text{\normalfont def}}

\newcommand{\M}[2]{{\mathcal M^{#1}_{#2}}}
\newcommand{\Mh}[2]{{\overline{\mathcal M}^{#1}_{#2}}}

\newcommand{\Cat}{\mathrm{Cat}}
\newcommand{\Top}{\mathrm{Top}}
\newcommand{\Diff}{\mathrm{Diff}}
\newcommand{\PL}{\mathrm{PL}}

\numberwithin{equation}{section}

\begin{document}

\title{Connected sum decompositions of high-dimensional manifolds}

\author[Bokor]{Imre Bokor}
\address{13 Holmes Avenue, Armidale, NSW 2350, Australia}
\email{ibokor@bigpond.net.au}

\author[Crowley]{Diarmuid Crowley}
\address{School of Mathematics and Statistics, University of Melbourne, Parkville, VIC, 3010, Australia}
\email{dcrowley@unimelb.edu.au}

\author[Friedl]{Stefan Friedl}
\address{Fakult\"at f\"ur Mathematik\\ Universit\"at Regensburg\\   Germany}
\email{sfriedl@gmail.com}

\author[Hebestreit]{Fabian Hebestreit}
\address{Rheinische Friedrich-Wilhelms-Universit\"at Bonn, Mathematisches Institut,
Endenicher Allee 60, 53115 Bonn, Germany}
\email{f.hebestreit@math.uni-bonn.de}

\author[Kasprowski]{Daniel Kasprowski}
\address{Rheinische Friedrich-Wilhelms-Universit\"at Bonn, Mathematisches Institut,
Endenicher Allee 60, 53115 Bonn, Germany}
\email{kasprowski@uni-bonn.de}

\author[Land]{Markus Land}
\address{Fakult\"at f\"ur Mathematik\\ Universit\"at Regensburg\\   Germany}
\email{markus.land@mathematik.uni-regensburg.de}

\author[Nicholson]{Johnny Nicholson}
\address{Department of Mathematics, UCL, Gower Street, London, WC1E 6BT, United Kingdom}
\email{j.k.nicholson@ucl.ac.uk}

\maketitle

\begin{abstract}
The classical Kneser-Milnor theorem says that every closed oriented connected 3-dimensional manifold admits a unique connected sum decomposition into mani\-folds that cannot be decomposed any further.  We discuss to what degree such decompositions exist in higher dimensions and we show that in many settings uniqueness fails in higher dimensions. 
\end{abstract}

\section{Introduction}
Consider the set $\M{\Cat}{n}$ of $n$-dimensional, oriented $\Cat$-isomorphism classes of $\Cat$-mani\-folds, where $\Cat = \Top, \PL$ or $\Diff$; unless explicitly stated, all manifolds are assumed non-empty, closed, connected and oriented. $\M{\Cat}{n}$ forms a monoid under connected sum (see Section~\ref{section:connected-sum})
as do its subsets
\[ \M{\Cat,\sc}{n}\,=\, \{M\in \M{\Cat}{n}\,|\, M\mbox{ is simply connected}\}\]
and
\[ \M{\Cat,\high}{n}\,=\, \{M\in \M{\Cat}{n}\,|\, M\mbox{ is highly connected}\};\]
here an $n$-manifold  $M$ is called \emph{highly connected} if $\pi_i(M)=0\mbox{ for }i=0,\dots,\lfloor \ttmfrac{n}{2}\rfloor-1$. Recall that the theorems of Rad\'o \cite{Rad26} and Moise \cite{Mo52,Mo77} show $\M{\Top}{n} = \M{\PL}{n} = \M{\Diff}{n}$ for $n \leq 3$, and by Cerf's work~\cite[p.~IX]{Ce68}  that  $\M{\PL}{n} = \M{\Diff}{n}$ for $n \leq 6$.
These monoids are countable. For $\PL$ and $\Diff$ this follows from the fact that triangulations exist.
For $\Top$ this follows from work of Cheeger and Kister~\cite{CK70}.

In this paper we want to study the question whether or not these monoids are unique factorisation monoids. 
First we need to make clear what we mean by a unique factorisation monoid.

\begin{definition}
\mbox{}
\bnm
\item Let $\MM$ be an abelian monoid (written multiplicatively). We say $m\in M$ is \emph{prime} if $m$ is not a unit
and if it divides a product only if it divides one of the factors.
\item  Given a monoid $\MM$ we denote by $\MM^*$ the units of $\MM$. We write $\overline{\MM}:=\MM/\MM^*$. 
\item Let $\MM$ be an abelian monoid. We denote by $\PP(\MM)$ the set of prime elements in $\overline{\MM}$. We say $\MM$ is a \emph{unique factorisation monoid} if the canonical monoid morphism $\N^{\PP(\MM)}\to \ol{\MM}$ is an isomorphism. 
\enm
\end{definition}

In dimensions $1$ and $2$ we of course have $\M{\Cat}{1} =\Mh{\Cat}{1}= \{[S^1]\}$ and $\M{\Cat}{2} =\Mh{\Cat}{2} \cong \N$ via the genus. In particular these monoids are unique factorisation monoids. In dimension $3$ there is the celebrated \emph{prime decomposition theorem}
which was stated and proved, in rather different language, by Kneser \cite{Kn29} and Milnor \cite{Mi62}. See also \cite[Chapter~3]{He76} for a proof. 

\begin{theorem}[Kneser-Milnor]\label{thm:kneser-milnor}\mbox{}
\bnm[font=\normalfont]
\item The monoid $\M{\Cat}{3}$ has no non-trivial units.
\item The monoid $\M{\Cat}{3}=\Mh{\Cat}{3}$ is a unique factorisation monoid.
\enm
\end{theorem}

The purpose of the present note is to study to what degree these statements hold in higher dimensions.

First, note that  all units of $\M{\Cat}{n}$ are homotopy spheres, as we deduce from an elementary complexity argument in Proposition~\ref{cor:units} below. It now follows from various incarnations of the Poincar\'e conjecture that $\M{\Top}{n}$ never has non-trivial units, and neither does $\M{\PL}{n}$, except potentially if $n = 4$. In the smooth category the current status is the following: The only odd dimensions in which $\M{\Diff}{n}$ has no non-trivial units are $1,3,5$ and $61$, see \cite{HHR,WX}. In even dimensions greater than $5$ Milnor and Kervaire construct an isomorphism $\Theta_n \cong \mathbb S_n/J_n$ from the group of smooth homotopy $n$-spheres $\Theta_n$ to the cokernel of the stable $J$-homomorphism $\pi_n(\mathrm{SO}) \rightarrow \mathbb S_n$, where $\mathbb S$ is the sphere spectrum. In dimensions below $140$, the only even dimension where $J_n$ is surjective are $2,4,6,12$ and $56$, see \cite{BHHM}. Wang and Xu recently conjectured that $1,2,3,$ possibly $4$, and $5,6,12,56, 61$ are the only dimensions without exotic spheres \cite{WX}.

Let us also mention that in dimension $4$ neither is it known whether every homotopy sphere is a unit nor whether $S^4$ is the only unit in $\M{\Diff}{4}$ (and of course these questions combine into the smooth $4$-dimensional Poincar\'e conjecture).

Before we continue we introduce the following definitions.

\begin{definition}
Let $\MM$ be an abelian monoid with neutral element $e$. 
\bnm
\item 
Two elements $m,n \in \MM$ are called \emph{associated} if there is a unit $u \in \MM^*$, such that $m = u\cdot n$.
\item  An element $m$ is called \emph{irreducible}, if it is not a unit and if all its divisors are associated to either $e$ or $m$.
\item An element $a$ is \emph{cancellable}, if $ab = ac$ implies $b=c$ for all elements $b,c \in \MM$.
\enm
\end{definition}

We make four remarks regarding these definitions.
\bnm
\item We warn the reader that for the monoids $\M{\Cat}{3}$ our usage of ``irreducible'' does \emph{not} conform with standard use in $3$-dimensional topology. More precisely, in our language $S^1\times S^2 \in \M{\Cat}{3}$ is irreducible, whereas in the usual language used in 3-dimensional topology, see \cite[p.~28]{He76}, the manifold $S^1\times S^2$ is not irreducible. Fortunately \cite[Lemma~3.13]{He76} says that this is the only 3-dimensional manifold for which the two definitions of irreducibility diverge (in fact by the prime decomposition theorem our notion of irreducible 3-dimensional manifold coincides with the usual use of the term prime 3-manifold).
\item Let $\MM$ be an abelian monoid. If $\MM$ is a unique factorisation monoid, then every element in $\ol{\MM}$ is cancellable.
\item
Another warning worthy of utterance is that in general, given a monoid $\MM$, neither all irreducible elements are prime, nor does a prime  element need to be irreducible, unless it is also cancellable.
\item Finally, note, that if $\MM$ is a unique factorisation monoid, this does not necessarily imply that every element in $\MM$ is cancellable: A good example is given by non-zero integers under multiplication modulo the relation that $x \sim -x$ if $|x| \geq 2$. In this case $\ol{\MM}\cong \N_{\geq 1}$ under multiplication, so $\MM$ is a unique factorisation monoid by prime decomposition. On the other hand we have $2\cdot (-1)=2\cdot 1$ and $-1\ne 1$. Thus we see that $2$ is not cancellable.
\enm

Using a fairly simple complexity argument we obtain the following result (Corollary~\ref{cor:infinite-divisor-chain} below).

\begin{proposition}\label{prop:decomposition-irreducibles-intro}
Every element in $\M{\Cat}{n}$ admits a connected sum decomposition into a homotopy sphere and irreducible manifolds.
\end{proposition}

\noindent
Unless $n = 4$ and $\Cat = \Diff$ or $\PL$, the homotopy sphere can of course, by the resolution of the Poincar\'e Conjecture, be absorbed into one of the irreducible factors.

As an example of the failure of cancellation and unique factorisation consider the manifolds $\cp^2\# \overline{\cp^2}$ and $S^2 \times S^2$. The intersection forms show that these manifolds are not homotopy equivalent, but it is well-known that $(S^2 \times S^2) \# \cp^2$ and $\cp^2 \# \overline{\cp^2} \# \cp^2$ are diffeomorphic \cite{GS99}. This implies easily that $\cp^2$ is not cancellable in $\Mh{\Cat}{4}$ and thus that $\M{\Cat}{4}$ is not a unique factorisation monoid.

The following is the main result of Section \ref{section:proof-non-simply-connected-case}:

\begin{theorem}\label{cor:not-cancellable}
For $n\geq 4$ the manifold $S^2\times S^{n-2}$ is not cancellable in any of the monoids $\Mh{\Cat}{n}$ and thus none of the monoids $\M{\Cat}{n}$ is a unique factorisation monoid in that range.
\end{theorem}

The proof we provide crucially involves manifolds with non-trivial fundamental groups. This leaves open the possibility that the submonoid $\Mh{\Cat,\sc}{n}$ consisting of simply connected $\Cat$-manifolds is better behaved. However, we show, for most dimensions, in Section \ref{section:proofs-simply-connected-case} that this is still not the case:

\begin{theorem}\label{thm:simply-connected-case}
For $n \ge 17$, the manifold $S^{5}\times S^{n-5}$ is not cancellable in any of the monoids $\Mh{\Cat,\sc}{n}$ and thus $\M{\Cat,\sc}{n}$ is not a unique factorisation monoid in that range.
\end{theorem}

The bound $n \geq 17$ is by no means intrinsic for finding non-cancellative elements in $\Mh{\Cat,\sc}{n}$; we already gave the example of $\cp^2 \in \Mh{\Cat}{4}$, and indeed by Wall's classification \cite{Wa62} the element $S^{2n} \times S^{2n}$ is non-cancellable in $\Mh{\Cat,\high}{4n}$, the monoid of highly connected $4n$-manifolds, once $n > 1$. 

Interestingly, in some cases the monoids $\M{\Diff, \high}{n}$ are actually unique factorisation monoids.
More precisely, by \cite{Wa62} and \cite[Corollary~1.3]{Sm62} we have the following theorem.

\begin{theorem}[Smale, Wall]\label{thm:high-dimensional-prime-decomposition}
For $k \equiv 3,5,7 \ \mathrm{mod}\ 8$, and $k \neq 15, 31, 63$, half the rank of $H_k$ gives an isomorphism $\Mh{\Diff,\high}{2k} \cong \N$. In particular, $\M{\Diff,\high}{2k}$ is a unique factorisation monoid in these cases.
\end{theorem}

As a final remark consider for $n \neq 4$ the exact sequence 
\[ 0 \to \Theta_n \to \M{\Diff}{n} \to \Mh{\Diff}{n} \to 0 \]
of abelian monoids. In general this sequence does not admit a retraction $\M{\Diff}{n}\to \Theta_n$: It is well-known that there are smooth manifolds $M$ for which there exists a non-trivial homotopy sphere $\Sigma$ such that $M \# \Sigma \cong M$, i.e.\ where the inertia group of $M$ is non-trivial, \cite[Theorem 1]{Wilkens}, \cite[Theorem 1.1]{Ram16}. For instance one can consider $M= \mathbb{HP}^2$. The group of homotopy $8$-spheres is isomorphic to $\Z/2$ and equals the inertia group of $\mathbb{HP}^2$.
Applying a potential retraction of the above sequence to the equation $[M \# \Sigma] = [M]$ gives a contradiction as $\Theta_n$ is a group. Potentially, the above sequence might admit a splitting, but we will not investigate this further.

\begin{remark}
The 
topic of this paper
is related to the notion of knot factorization.
More precisely, Schubert \cite{Sch49} showed in 1949 that the monoid of oriented knots in $S^3$, where the operation is given by connected sum, is a unique factorisation monoid.
It was shown by Kearton \cite{Ke79,Ke80} and Bayer \cite{Bay80} that the higher-dimensional analogue does \emph{not} hold.
\end{remark}

\begin{remark}
The question of whether a given $n$-manifold $M$ is reducible is in general a hard problem and the answer
can depend on the category.  When $n \geq 3$ and $M$ is a $j$-fold connected sum $M  = M_1 \# \dots \# M_j$, then 
$\pi_1(M) \cong \pi_1(M_1) \ast \dots \ast \pi_1(M_j)$ is the free product of the fundamental groups of the summands.  The converse of this
statement goes by the name of the Kneser Conjecture.   
When $n = 3$, the Kneser Conjecture was proved by Stallings in his PhD thesis, see also \cite[Theorem~7.1]{He76}.
  In higher dimensions, results of Cappell
showed that the Kneser Conjecture fails \cite{Ca74, Ca76} and when $n = 4$, 
Kreck, L\"uck and Teichner
\cite{KLT95} showed that the Kneser Conjecture fails in both $\M{\Diff}{4}$ and $\M{\Top}{4}$ 
and even gave an example of an irreducible smooth $4$-manifold which is topologically reducible.
\end{remark}

\subsection*{Organisation}
The paper is organized as follows. 
In Section~\ref{section:connected-sum} we study the behaviour of complexity functions under the connected sum operation and use the results to provide the proof of Proposition~\ref{prop:decomposition-irreducibles-intro} and the characterisation of the units.
In Section~\ref{section:wall-hc} we extract some results from Wall's classification of highly-connected manifolds. 
In Section~\ref{section:thickenings} we recall Wall's thickening operation which makes it possible to associate manifolds to CW-complexes. 
In Section~\ref{section:proof-non-simply-connected-case} we show the existence of interesting pairs of 2-dimensional CW-complexes which allow us to prove Theorem~\ref{thm:no-unique-prime-decomposition-high-dimension}
and~\ref{thm:no-unique-prime-decomposition-dimension-4}.
Furthermore in Section \ref{section:proofs-simply-connected-case} we recall the construction of interesting pairs of 8-dimensional simply connected CW-complexes which leads to the proof of Theorem~\ref{thm:simply-connected-case}.
In Section~\ref{section:prime-manifolds} we discuss the existence of prime manifolds in various monoids, in particular in we show that the Wu manifold $W=\op{SU}(3)/\op{SO}(3)$ is prime in $\M{\Cat,\sc}{5}$.
Finally in Section~\ref{section:questions} we list some open problems.

\subsection*{Acknowledgments.}
Most of the work on this paper happened while the authors attended the workshop ``Topology of Manifolds: Interactions Between High And Low Dimensions'' that took place January 7th-18th at the mathematical research institute MATRIX in Creswick, Australia.  We are very grateful to MATRIX for providing an excellent research environment.

SF and ML were supported by the SFB 1085 ``Higher Invariants'' at the University of Regensburg  funded by the DFG, FH and DK were funded by the Deutsche Forschungsgemeinschaft (DFG, German Research Foundation)
under Germany's Excellence Strategy - GZ 2047/1, Projekt-ID 390685813.
JN was supported by the UK Engineering and Physical Sciences Research Council (EPSRC) grant EP/N509577/1.

We wish to thank Mark Powell, Manuel Krannich, and Patrick Orson for helpful conversations.
We also  wish to thank the referee for several useful comments. In particular the proof of
Proposition~\ref{prop:borel-conjecture-implies-unique-dec} was suggested to us by the referee.

\section{The connected sum operation} \label{section:connected-sum}

We recall the definition of the connected sum. Let $n\in \N$ and let $M,N \in \M{\Cat}{n}$ be two $\Cat$-manifolds. Given an orientation-preserving $\Cat$-embedding $\varphi\colon \overline{B}^n\to M$ and given an orientation-reversing $\Cat$-embedding  $\psi\colon \overline{B}^n\to N$ we define the \emph{connected sum} of $M$ and $N$ as 
\[
\hspace{1.2cm} M\# N\,\,:=\,\, \big(M\setminus \varphi\big(B^n\big)\big)\,\sqcup\, 
\big(N\setminus \psi\big(B^n\big)\big)\,/\, \varphi(P)=\psi(P)\mbox{ for all } P\in S^{n-1},\]
given a smooth structure (for $\Cat = \Diff$) by rounding the corners, and a piecewise linear one (for $\Cat = \PL$) by choosing appropriate triangulations that make the images of $\varphi$ and $\psi$ sub-complexes.

For $\Cat = \Diff$ or $\PL$ the fact that the isomorphism type of the connected sum of two manifolds does not depend on the choice of embedding is a standard fact in (differential) topology, see e.g.\ \cite[Theorem~2.7.4]{Wa16} and \cite[Disc Theorem~3.34]{RS72}. The analogous statement also holds for $\Cat=\Top$,
but the proof (for  $n > 3$) is significantly harder. It follows in a relatively straightforward way from the ``annulus theorem'' that was proved in 1969 by Kirby~\cite{Ki69} for $n\ne 4$ and in 1982 by Quinn~\cite{Qu82,Ed84} for $n=4$. We refer to \cite{FNOP19} for more details.

It is well-known that many invariants are well-behaved under the connected sum operation. 
Before we state the corresponding lemma that summarizes the relevant results we introduce our notation for intersection forms and linking forms. 
Given a $2k$-dimensional manifold $W$ we denote by $Q_W\colon H_k(W;\Z)\times H_k(W;\Z)\to \Z$ the intersection form of $W$. Furthermore, given a $(2k+1)$-dimensional manifold $W$ we denote by $\lk_W\colon \op{Tors} H_k(W;\Z)\times \op{Tors} H_k(W;\Z)\to \Q/\Z$ the linking form of $W$.

\begin{lemma}\label{lem:invariants-additive}
If $M_1,\dots,M_k$ are $n$-dimensional manifolds,
then the following statements hold:
\bnm[font=\normalfont]
\item \label{it:fundamentalgroup} If $n\geq 3$, then $\pi_1(M_1\# \dots \# M_k)\cong \pi_1(M_1)*\dots*\pi_1(M_k)$,
\item \label{it:cupproduct}  Let $R$ be a ring. Then the cohomology ring $H^*(M_1 \# \dots \# M_k;R)$ is a quotient of a subring of the product $H^*(M_1;R) \times \dots \times H^*(M_k;R)$: First, consider the subring of this product where the elements in degree zero are in the image of the diagonal map $R \to R^k$ (recall that all $M_i$ are connected). Then divide out the ideal generated by $(\mu_i - \mu_j)$ for $1 \leq i,j \leq k$, where $\mu_i$ is the cohomological fundamental class of $M_i$,
\item if $n$ is even, then
$Q_{M_1\# \dots \# M_k}\cong Q_{M_1}\oplus \dots \oplus Q_{M_k}$,
\item if $n$ is odd, then
$\lk_{M_1\# \dots \# M_k}\cong \lk_{M_1}\oplus \dots \oplus \lk_{M_k}$.
\enm
\end{lemma}

\begin{proof}
By induction, it suffices to treat the case $k=2$. (1) is an elementary application of the Seifert--van Kampen theorem. (2) follows from the cofibre sequence
$$S^{n-1} \to M_1 \# M_2 \to M_1 \vee M_2 \to S^n:$$ 
Recall that the subring described in the statement is precisely the cohomology of $M_1 \vee M_2$. It is then easy to see that the map $M_1 \vee M_2 \to S^n$ induces an injection on $H^n(-;R)$ with image precisely $\mu_1 - \mu_2$.
Statements (3) and (4) follow from (2) and the fact that the described isomorphism of rings is compatible with the Bockstein operator.
\end{proof}

One can extract a crude complexity invariant from the above data, as follows. Given a finitely generated abelian group $A$ we denote by  $\op{rank}(A) =\dim_{\Q}(A\otimes \Q)$ its rank and we define
\[ t(A)\,\,:=\,\,\ln(\# \mbox{torsion subgroup  of $A$}).\]
Furthermore, given a finitely generated group we denote by $d(G)\in \N$ the minimal number of elements in a generating set for $G$.

Let $M$ be an $n$-dimensional topological manifold. 
Then $M$ is a retract of a finite CW-complex \cite[p.~538]{Bre93}, and thus has the homotopy of a CW complex (e.g. \cite[Proposition A.11]{Ha01}) and its fundamental group and the (co-) homology groups  of $M$ are finitely presented. Thus we can
define the \emph{complexity of $M$} as 
\[ c(M)\,\,:=\,\, d(\pi_1(M))\,+\, \op{rank}\Big(\,\ttmoplus{i=1}{n-1} H_i(M;\Z)\Big)\,+\, 
 t\Big(\,\ttmoplus{i=1}{n-1} H_i(M;\Z)\Big)\,\in\, \R_{\geq 0}.\]
The following proposition summarizes two key properties of $c(M)$.

\begin{proposition}\label{prop:prop-c}
\mbox{}
\bnm[font=\normalfont]
\item For $n \geq 3$ the complexity gives a homomorphism $\M{\Cat}{n} \rightarrow \R_{\geq 0}$.
\item The kernel of $c$ consists entirely of homotopy spheres.
\enm
\end{proposition}

For the proof, recall the Grushko-Neumann Theorem which is proved in most text books on combinatorial group theory, e.g.\ \cite[Corollary~IV.1.9]{LS77}.

\begin{theorem} \textbf{\emph{(Grushko-Neumann Theorem)}}\label{thm:grushko-2}
Given any two finitely generated groups $A$ and $B$ we have 
\[ d(A*B)\,\,=\,\, d(A)+d(B).\]
\end{theorem}

\begin{proof}[Proof of Proposition \ref{prop:prop-c}]
The first statement follows from Lemma~\ref{lem:invariants-additive} and the Grushko-Neumann Theorem~\ref{thm:grushko-2}. The second statement follows since by the Hurewicz Theorem 
we have $\pi_n(M) \cong \Z$ if $c(M) = 0$, and a generator of $\pi_n(M)$ 
is represented by a map $S^n \rightarrow M$ which is a homotopy equivalence by Whitehead's theorem (and the observation above, that $M$ has the homotopy type of a CW-complex).
\end{proof}

\begin{proposition}\label{cor:units}\mbox{}
\bnm[font=\normalfont]
\item 
All units of $\M{\Cat}{n}$ are homotopy spheres.
\item The converse to \textup{(1)} holds if $n\ne 4$ or if $n=4$ and $\Cat=\Top$.
\item The neutral element is the only unit in  $\M{\Top}{n}$ and $\M{\PL}{n}$.
\enm
\end{proposition}

By the work of Smale, Newman, Milnor and Kervaire the groups of homotopy spheres are of course relatively well understood for $n \geq 5$.

\begin{proof}\mbox{}
\bnm
\item 
Since units are mapped to units under homomorphisms  it follows from
Proposition~\ref{prop:prop-c} (2) that units are homotopy spheres.
\item 
For $n \leq 2$ the classification of $n$-manifolds clearly implies the converse to (1).
In dimensions $3$ and $4$ the desired result follows straight from the Poincar\'e conjecture proved by Perelman and Freedman. 
Now let  $n \geq 5$.  Given an $n$-dimensional homotopy sphere $\Sigma$ and an $n$-disc $D \subseteq \Sigma^n$, $(\Sigma\setminus \mathrm{int(D)}) \times I$ with an open $n+1$-disc removed away from the boundary is an h-cobordism between $\Sigma \#\overline\Sigma$ and $S^n$. For $n \neq 5$ the h-cobordism theorem then implies that $\Sigma$ is indeed a unit. For $n=5$, every homotopy sphere is h-cobordant to the standard sphere, see e.g.\ \cite[Chapter~X~(6.3)]{Kosinski}, and hence diffeomorphic to $S^5$ by the h-cobordism theorem.
\item This result follows from the resolution of the Poincar\'e Conjecture.\qedhere
\enm
\end{proof}

The following corollary implies in particular Proposition~\ref{prop:decomposition-irreducibles-intro}.

\begin{corollary}\label{cor:infinite-divisor-chain}
Unless $n=4$ and $\Cat = \Diff$ or $\PL$, the monoid $\M{\Cat}{n}$ does not admit infinite divisor chains, and therefore every manifold admits a decomposition into irreducible manifolds.
\end{corollary}

\begin{proof}
An infinite divisor chain $M_i$ gives rise to a descending sequence of natural numbers under $c$, which becomes stationary after index $i$ say. But then for $j \geq i$ the elements witnessing that $M_j$ is a summand of $M_i$ have vanishing complexity and thus units by Proposition~\ref{cor:units}, so $M_j$ is associated to $M_i$ for all $j \geq i$.
\end{proof}

Similar arguments also allow us to identify some irreducible elements of $\M{\Cat}{n}$.

\begin{corollary}\label{cor:examples-topologically-irreducible}
The manifolds $\rp^{2n-1},\cp^n,\HP^n,\OP^2$ and $S^n\times S^{k-n}$ are irreducible in the monoid  $\M{\Cat}{m}$ for any choice of $\Cat$ and appropriate dimension $m$, except possibly for  $\M{\Diff}{4}$.
\end{corollary}

\begin{proof}
This follows immediately from \cref{lem:invariants-additive}~\eqref{it:fundamentalgroup} and \eqref{it:cupproduct}. 
\end{proof}

\section{Wall's work on highly connected manifolds}\label{section:wall-hc}
The possibility of prime factorisations in $\Mh{\PL,\high}{2k}$ and $\Mh{\Diff,\high}{2k}$ was studied by Wall in \cite[Problem~2A]{Wa62}. He classified such smooth manifolds in terms of their intersection form and an additional invariant $\alpha \colon H_k(M) \rightarrow \pi_k\mathrm{BSO}(k)$, which is given by representing an element in $H_k(M)$ by an embedded sphere and taking its normal bundle. In the case of piecewise linear manifolds Wall restricts attention to manifolds that can be smoothed away from a point; for even $k$, the map $\alpha$ is then well-defined (i.e. independent of the chosen smoothing) by \cite[Lemma 2 \& Formula (13)]{Wa62} and the injectivity of the stable $J$-homomorphism (which was not known at the time). For $k$ odd, one furthermore has to invest the injectivity of the unstable $J$-homomorphism $\pi_{k}(\mathrm{BSO}(k)) \rightarrow \pi_{2k-1}(S^k)$.

\begin{theorem}[Wall] \label{thm:wall-highly-connected}
	Unique factorisation in $\Mh{\PL,\high}{2k}$ holds only for $k=1,3$ and possibly $k = 7$. In addition to these cases unique factorisation in $\Mh{\Diff,\high}{2k}$ holds exactly for $k\equiv 3,5,7\mod 8$ with $k\neq 15,31$ and possibly $k\neq 63$ $($if there exists a Kervaire sphere in dimension $126$\textup{)}.
\end{theorem}

In all cases that unique factorisation holds, the monoid in question is actually isomorphic to $\N$ via half the rank of the middle homology group, except possibly $\Mh{\PL,\high}{14}$. In fact there does not seem to be a full description of $\Mh{\PL,\high}{n}$ (or $\Mh{\Top,\high}{n}$) in the literature. Let us remark, that the work of Kirby-Siebenmann implies $\M{\PL,\high}{n} \cong \M{\Top,\high}{n}$, once $n \geq 10$, as the obstruction to finding a $\PL$-structure on a topological manifold $M$ is located in $H^4(M;\mathbb Z/2)$, with $H^3(M;\mathbb Z/2)$ acting transitively on isotopy classes of $\PL$-structures. Wall's argument also shows that unique factorisation fails in $\Mh{\Top,\high}{8}$.

\begin{proof}	
	It follows from Wall's work that for $k \geq 4$ even, the monoids $\Mh{\PL,\high}{2k}$ and $\Mh{\Diff,\high}{2k}$
	never admit unique factorisations; this can be seen by picking an even positive definite unimodular form, and realizing it by a $(k-1)$-connected $2k$-dimensional manifold $M$ with $\alpha$-invariant whose values lie in $\mathrm{ker}(\pi_k\mathrm{BSO}(2k) \rightarrow \pi_k\mathrm{BSO})$; such $\alpha$ is uniquely determined by the intersection form by \cite[Lemma 2]{Wa62} and the computation of $\pi_k\mathrm{BSO}(2k)$ on \cite[p.~171]{Wa62}. By  \cite[Proposition~5]{Wa62}, in this case a  smooth realizing manifold exists whenever the signature is divisible by a certain index. Then $M\#-M\cong m(S^k\times S^k) \# \Sigma$ for some homotopy sphere $\Sigma$, where $m$ is the rank of $H_k(M)$. But $S^k\times S^k$ cannot divide $M$ or $-M$. Indeed, this follows from 
	Lemma~\ref{lem:invariants-additive} (3), the fact that the intersection form of $S^k\times S^k$ is indefinite and the fact that the intersection forms of $\pm M$ are definite. See Proposition~\ref{prop:4dim} below for a stronger statement in the case $k=2$.
	
	For odd values of $k$ there are several cases to be distinguished. To start, for $k=1,3$ and $\Cat = \PL$ or $k=1,3,7$ and $\Cat = \Diff$ the monoid $\Mh{\Cat,\high}{2k}$ is isomorphic to $\N$ via half the rank of the middle homology by \cite[Lemma~5]{Wa62}.
	
	For other odd values of $k\neq 1,3,7$ unique decomposition in $\M{\PL,\high}{2k}$ never holds. This can be seen via the Arf-Kervaire invariant; this is the Arf invariant of a certain quadratic refinement of the intersection form. By the work of Jones and Rees and Stong \cite{Jones, Stong} any highly connected manifold of even dimension not $2,4,8$ or $16$ possesses a canonical such refinement. We proceed by taking a manifold $M$ which is smoothable away from a point with non-trivial Arf-Kervaire invariant and then decompose $M\# M$ into manifolds with vanishing Kervaire invariant and intersection form hyperbolic of rank $2$; this is possible by \cite[Lemmata~5 and~9]{Wa62} and the fact that the Arf-Kervaire invariant is additive.
	
	For $\M{\Diff, \high}{2k}$ the argument above works equally well if there exists a smooth $2k$-manifold with Kervaire invariant one (which also implies the existence of an irreducible one by Corollary~\ref{cor:infinite-divisor-chain}). This famously is the case if and only if $k=1,3,7,15,31$ and possibly $k=63$  \cite{HHR}, which rules out unique factorisation in these dimensions. 
	
	For $k = 3,5,7\mod 8$ with $k\neq 3,7,15,31,63$ the monoid $\Mh{\Diff,\high}{2k}$ is in fact isomorphic to $\N$ via half the rank of the middle homology by \cite[Lemma~5]{Wa62} (with the case $\Mh{\Diff,\high}{126}$ being open). In contrast, for $k\equiv 1\mod 8$ the failure of unique decomposability can be seen by considering the composite homomorphism \[S\alpha \colon H_k(M) \xrightarrow{\alpha} \pi_k\mathrm{BSO}(k) \longrightarrow \pi_k\mathrm{BSO} = \mathbb Z/2:\]
Wall says that a manifold is of type 0 if $S\alpha$ is non-trivial and of type 1 if $S\alpha$ is trivial, see \cite[p.~173,~Case~5]{Wa62}. Note that the type is not additive but the connected sum of two manifolds has type 1 if and only if both manifolds have type 1. Hence by \cite[Theorem~3]{Wa62} we can pick any two irreducible manifolds $W_0,W_1$ such the invariants from \cite[Lemma~5]{Wa62} agree except that the type of $W_i$ is $i$. Then $W_0\#W_1\cong W_0\#W_0$ and unique decomposition fails.
\end{proof}

\begin{remark}Given the first part of the proof above one might wonder whether for any highly connected manifold $M$, whose intersection form is even, $M \# \ol{M}$ is homeomorphic to $\#^k (S^n \times S^n)$. This is in fact not correct as the following example shows. Let $M$ be the total space of an $S^4$-bundle over $S^4$ with non-trivial first Pontryagin class and trivial Euler class. It is then easy to see that the intersection form of $M$ is even. However, since the rational Pontryagin classes are homeomorphism invariants we find that $M \#\ol{M}$ is not homeomorphic to $S^4\times S^4 \# S^4 \times S^4$.
\end{remark}

Cancellation in $\M{\PL,\high}{2k}$ and $\M{\Diff,\high}{2k}$ was also studied by Wall in \cite[Problem~2C]{Wa62}.
\begin{theorem}[Wall]\label{nchc}
	Cancellation in $\Mh{\Diff,\high}{2k}$ holds if and only if either $k=1$ or $k\equiv 3,5,7\mod 8$. In $\M{\PL,\high}{2k}$ the ``only if'' part still holds.
\end{theorem}

\begin{proof}
	Wall's classification directly shows that $S^k \times S^k$ is not cancellable in either $\M{\PL,\high}{2k}$ and $\M{\Diff,\high}{2k}$ if $k$ is even by an argument similar to the one above. 
	For $k=2$, the failure of cancellation follows from the fact that $(S^2\times S^2)\#\cp^2$ and $\cp^2\#\overline{\cp}^2\#\cp^2$ are diffeomorphic.
	
	By \cite[Lemma~5]{Wa62} and the classification of almost closed $n-1$-connected $2n$-manifolds by their $n$-types (see \cite[page 170]{Wa62}), for $k\equiv 3,5,7\mod 8$ the monoid $\M{\Diff,\high}{2k}$ embeds into $\Z\times\Z/2$ via the rank and the Arf-Kervaire invariant. If $k\equiv 1 \mod 8$, $k \neq 1$, then the examples from the previous proof exhibit the failure of cancellation. 
\end{proof}

A similar analysis of $\Mh{\Cat,\high}{2k+1}$ can be carried out using \cite{Wall67}, but we refrain from spelling this out here.

\section{Thickenings of finite CW-complexes}\label{section:thickenings}
In this section we will see that one can associate to a finite CW-complex a smooth manifold
which is unique in an appropriate sense. This procedure allows us to translate information about CW-complexes to manifolds. We will use this procedure in  the proofs of Theorems \ref{thm:no-unique-prime-decomposition-high-dimension},
\ref{thm:no-unique-prime-decomposition-dimension-4} and \ref{thm:simply-connected-case} in the following sections.

\begin{convention}
	By a \emph{finite complex} we mean a finite connected CW-complex and by a \emph{finite $n$-complex} we mean a finite connected $n$-dimensional CW-complex.
\end{convention}

\subsection{Thickenings of finite CW-complexes}
In this section we will summarize the theory of smooth thickenings of CW-complexes as developed in \cite{Wa66}. As is explained in \cite{Wa66} there is also an analogous theory of PL-thickenings.

We start out with the following definition, which is an adaptation of the definition on  \cite[p.~74]{Wa66} for our purposes. 

\begin{definition}
Let $X$ be a finite complex.
\bnm
\item  A \emph{$k$-thickening of $X$} is a pair $(M,\phi)$ where $M$ is an oriented, smooth, compact $k$-dimensional manifold with trivial tangent bundle, which has the property that the inclusion induced  map $\pi_1(\partial M)\to \pi_1(M)$ is an isomorphism and where $\phi\colon X\to M$ is a simple homotopy equivalence. 
\item Two $k$-thickenings $(M,\phi)$ and $(N,\psi)$ of $X$ are called \emph{equivalent} if there exists an orientation-preserving diffeomorphism $f\colon M\to N$ such that $f\circ \phi$ is homotopic to $\psi$.
\enm
\end{definition}

Note that $\phi$ does not have to be an embedding in the definition above. 

\begin{theorem}\label{thm:thickenings-exist}
	\label{prop:unique-thickening}
Let $X$ be a finite $n$-complex.
\begin{enumerate}[font=\normalfont]
	\item If $k\geq 2n$, then there exists a $k$-thickening of $X$. 
	\item If $k\geq 2n+1$ and $k\geq 6$,
then all $k$-thickenings of $X$ are equivalent.
\end{enumerate}
\end{theorem}
\begin{proof}
	The theorem follows from \cite[p.~76]{Wa66} using that every map from a path-connected space to $\R^k$ is 1-connected.
\end{proof}
%
%
%

\begin{definition}
Let $k\geq 2n+1$ and $k\geq 6$.
Let $X$ be a finite $n$-complex. We denote by $N^k(X)$ the oriented diffeomorphism type of the $k$-dimensional thickening of $X$. In our notation we will not distinguish between $N^k(X)$ and any representative thereof.
\end{definition}

For convenience we state the following example.

\begin{lemma} \label{lem:thicken-sphere-n}
If $k \ge 2n+1$ and $k\geq 6$, then $N^k(S^n) = S^n \times B^{k-n}$.
\end{lemma}

\begin{proposition}\label{prop:unique-thickening-simple-homotopic}
Let $X$ and $Y$ be finite complexes. We suppose that $k\geq 2\dim(X)+1$, $k\geq 2\dim(Y)+1$  and $k\geq 6$. If $X$ and $Y$ are simple homotopy equivalent, then there exists an orientation-preserving diffeomorphism from  $N^k(X)$ to $N^k(Y)$. 
\end{proposition}

\begin{proof}
Let $f\colon X\to Y$ be a simple homotopy equivalence. 
Let $(M,\phi)$ be a $k$-thickening for $X$ and let $(N,\psi)$ be a $k$-thickening for $Y$.
Note that $(N,\psi\circ f)$ is a $k$-thickening for $X$. It follows from Theorem~\ref{prop:unique-thickening}, and our dimension restrictions on $k$, that $N=N^k(X)$ is diffeomorphic to $M=N^k(Y)$. 
\end{proof}

\begin{lemma} \label{lem:thicken-sum-m}
Let $X$ and $Y$ be finite complexes. If   $k\geq 2\dim(X)+1$, $k\geq 2\dim(Y)+1$  and $k\geq 6$, then 
\[ N^k(X\vee Y)\,\,=\,\, N^k(X) \#_b N^k(Y),\]
where ``$\#_b$'' denotes the boundary connected sum.
\end{lemma}

\begin{proof}
Let $(M,\phi)$ and $(N,\psi)$ be $k$-thickenings of $X$ and $Y$, respectively.  
After a simple homotopy we can and will assume that the the image of the wedge points under $\phi$ and $\psi$ lies on the boundary of $M$ and $N$. 
Note that  $\phi\vee \psi\colon X\vee Y\to M\vee N$ is a simple homotopy equivalence
and note that the inclusion $M\vee N\to M\#_b N$ is a simple homotopy equivalence.
Thus we see that the map $\phi\vee \psi\colon X\vee Y\to M\#_b N$ is a simple homotopy equivalence. It follows almost immediately from 
Theorem~\ref{prop:unique-thickening} that $N^k(X\vee Y)=N^k(X) \#_b N^k(Y)$.
\end{proof}



\subsection{Boundaries of thickenings of finite complexes}

\begin{definition}
Let $X$ be a finite complex and let $k\geq 2\dim(X)$ and $k\geq 5$.
We write  $M^k(X):=\partial N^{k+1}(X)$. Recall that $N^{k+1}(X)$ is an oriented manifold 
and we equip $M^k(X)$ with the corresponding orientation.
\end{definition}

The following lemma is an immediate consequence of Lemma~\ref{lem:thicken-sphere-n}.

\begin{lemma} \label{lem:thicken-sphere}
If $k \ge 2n$ and $k\geq 5$, then $M^k(S^n) = S^n \times S^{k-n}$.
\end{lemma}

In the following proposition we summarize a few properties of $M^k(X)$. 

\begin{proposition}\label{prop:thicken-sum} \label{prop:thicken-homotopy}
For $n\in \N$ let $X$ and $Y$ be finite $n$-complexes. Furthermore let $k\in \N$ with  $k\geq 2n$ and $k\geq 5$. 
\bnm[font=\normalfont]
\item $M^k(X)$ is a closed oriented $k$-dimensional manifold,
\item if $X$ and $Y$ are simple homotopy equivalent, then there exists an orientation-preserving diffeomorphism from  $M^k(X)$ to $M^k(Y)$,  
\item $M^k(X\vee Y)=M^k(X)\# M^k(Y)$.
\enm
If we have in fact $k\geq 2n+1$, then the following also holds:
\bnm
\item[(4)] if $M^k(X)$ and $M^k(Y)$ are homotopy equivalent, then $X$ and $Y$ are homotopy equivalent.
\enm
\end{proposition}

\begin{proof}
The first statement follows immediately from the definitions, the second from Proposition~\ref{prop:unique-thickening-simple-homotopic} and the third is a straightforward consequence of  Lemma~\ref{lem:thicken-sum-m}.
The fourth statement is proved in \cite[Proposition~II.1]{KS84}.
\end{proof}

Let $n\in \N$. Furthermore let $k\in \N$ with  $k\geq 2n$ and $k\geq 5$. 
By Proposition~\ref{prop:thicken-homotopy} we obtain a well-defined map
\[ M^k \colon \{ \text{finite $n$-complexes}\}/\simeq_s \,\, \to\,\, \M{\Diff}{k}\]
where $\simeq_s$ denotes simple homotopy equivalence.

We conclude this section with the following corollary, which we will make use of in the proofs of 
Theorems \ref{thm:no-unique-prime-decomposition-high-dimension},
~\ref{thm:no-unique-prime-decomposition-dimension-4} and \ref{thm:simply-connected-case} respectively.

\begin{corollary}\label{thm:need-for-nonsimplyconnectedcase}
Let $n\in \N$. Furthermore let $k\in \N$ with  $k\geq 2n$ and $k\geq 5$.  Suppose that $X$ and $Y$ are finite complexes of dimension $\leq n$. 
We suppose that  $X \not \simeq Y$ and $X \vee S^n \simeq_s Y \vee S^n$. Then $M^k(X) \not \simeq M^k(Y)$, but there is an orientation preserving diffeomorphism between $M^k(X) \# (S^n \times S^{k-n})$ and $M^k(Y) \# (S^n \times S^{k-n})$.
\end{corollary}

\begin{proof}
This corollary follows immediately from the four statements of  Proposition~\ref{prop:thicken-homotopy}.
\end{proof}

\subsection{5-dimensional  thickenings}

Now let $X$ be a  finite $2$-complex. By Theorem~\ref{thm:thickenings-exist} there exists a $5$-thickening of $X$.
We can no longer conclude from Theorem~\ref{prop:unique-thickening} that the thickening is well-defined up to diffeomorphism. 
But in fact the following weaker statement holds:

\begin{proposition}\label{prop:unique-thickening-dim-4}
Let $X$ be a finite $2$-complex. 
If $(M,\phi)$ and $(N,\psi)$ are 5-dimensional thickenings for $X$, then $\partial M$ and $\partial N$ are $s$-cobordant.
\end{proposition}

This statement is implicit in \cite{Wa66}, see also \cite[p.~15]{KS84}. 
The same way that we deduced
Proposition~\ref{prop:unique-thickening-simple-homotopic} from 
Theorem~\ref{prop:unique-thickening} we can also deduce the following proposition
from Proposition~\ref{prop:unique-thickening-dim-4}.

\begin{proposition}\label{prop:unique-thickening-simple-homotopic-dim-4}
Let $X$ and $Y$ be finite $2$-complexes.
If $X$ and $Y$ are simple homotopy equivalent, then given any 5-dimensional thickenings $A$ of $X$ and $B$ of $Y$ the boundaries $\partial A$ and $\partial B$ are $s$-cobordant.
\end{proposition}

\section{Finite $2$-complexes, group presentations and the D2-problem}\label{section:proof-non-simply-connected-case}
The goal of this section is to prove Theorem \ref{cor:not-cancellable} from the introduction and to give a survey of the various constructions that can be used to construct examples of non-cancellation in $\M{\Cat}{n}$. We will use:

\begin{lemma}\label{lem:not-cancellable}
Let $n\in \N$. Suppose there exist $n$-dimensional smooth manifolds $M$ and $N$ which are not homotopy equivalent but such that there is $r\geq 1$, an $n$-dimensional smooth manifold $W$ and an orientation preserving diffeomorphism between $M\# r\cdot W$ and $N\# r\cdot W$.
Then 
for every $\Cat=\Top$, $\PL$ and $\Diff$  the following two statements hold:
\bnm[font=\normalfont]
\item The element $W$ is not cancellable in $\Mh{\Cat}{n}$.
\item The monoid $\M{\Cat}{n}$ is not a unique factorisation monoid.
\enm
\end{lemma}

\begin{proof}
By hypothesis we know that $[M_1]\ne [M_2]\in \M{\Top}{n}$.
By Proposition~\ref{cor:units} we know that $\M{\Top}{n}=\Mh{\Top}{n}$.
In particular $[M_1]\ne [M_2]\in \Mh{\Top}{n}$ and thus  $[M_1]\ne [M_2]\in \Mh{\Cat}{n}$.
Furthermore we know that $[M_1]+r\cdot [W]=[M_2]+r\cdot [W]\in \Mh{\Cat}{n}$. 
By induction we see that $[W]$ is not cancellable in $\Mh{\Cat}{n}$. This implies that 
$\Mh{\Cat}{n}$ is not isomorphic to some $\N^P$, i.e.\ $\Mh{\Cat}{n}$ is not a unique factorisation monoid, hence $\M{\Cat}{n}$ is not a unique factorisation monoid.
\end{proof}

We will exploit this lemma with the following result:

\begin{theorem}\label{thm:no-unique-prime-decomposition-high-dimension}
Let $n\in \N_{\geq 5}$. Then there exist $n$-dimensional smooth manifolds $M$ and $N$ which are not homotopy equivalent but such that there is an orientation preserving diffeomorphism between $M\# (S^2\times S^{n-2})$ and $N\# (S^2\times S^{n-2})$.
\end{theorem}

A slightly weaker result is also available in dimension $4$: 

\begin{theorem}\label{thm:no-unique-prime-decomposition-dimension-4}
There exist 4-dimensional smooth manifolds $M$ and $N$ which are not homotopy equivalent but such that there is an orientation preserving diffeomorphism between $M\# r\cdot (S^2\times S^{2})$ and $N\# r\cdot (S^2\times S^{2})$ for some $r\geq 1$. 
\end{theorem}

Taken together these results immediately imply Theorem \ref{cor:not-cancellable} from the introduction.

\subsection{Proof of Theorem \ref{thm:no-unique-prime-decomposition-high-dimension}} \label{subsection:proof}
In this section we will provide the proof for Theorem \ref{thm:no-unique-prime-decomposition-high-dimension}.
The key idea for finding suitable manifolds is to use Corollary~\ref{thm:need-for-nonsimplyconnectedcase}. Thus our goal is to find finite 2-complexes $X$ and $Y$ which are not homotopy equivalent but such that $X\vee S^2$ and $Y\vee S^2$ are simple homotopy equivalent.

We begin our discussion with the following well-known construction of a finite $2$-complex $X_{\PP}$ with $\pi_1(X_{\PP})=G$ from a group presentation
\[ \PP = \langle x_1, \cdots, x_s \mid r_1, \cdots, r_t \rangle\]
of a finitely presented group $G$. This is known as the Cayley complex  $X_{\PP}$ of the presentation $\PP$ and has $1$-skeleton a wedge of $s$ circles, one circle for each generator $x_i$, with its $2$-cells attached along the paths given by each relation $r_i$ expressed as a word in the generators.

The following can be found in \cite{Bro79}, \cite[Theorem~B]{HK93 I}:

\begin{theorem} \label{thm:finite-complexes-are-a-fork}
If $X$ and $Y$ are finite $2$-complexes with $\pi_1(X)\cong \pi_1(Y)$ finite and $\chi(X)=\chi(Y)$, then $X \vee S^2 \simeq_s Y \vee S^2$.	
\end{theorem}

Recall that, if $\PP$ is a presentation of $G$ with $s$ generators and $t$ relations, then the deficiency $\Def(\PP)$ of $\PP$ is $s-t$.  The Euler characteristic of a presentation complex can be completely understood in terms of the deficiency:

\begin{lemma} \label{lem:euler-char}
If $\PP$ is a group presentation, then $\chi(X_\PP) = 1 - \Def(\PP)$.
\end{lemma}

The task is therefore to find a finite group $G$ with presentations $\PP_1$ and $\PP_2$ such that $X_{\PP_1} \not \simeq X_{\PP_2}$ and $\Def(\PP_1)=\Def(\PP_2)$. That $X_{\PP_1} \vee S^2 \simeq_s X_{\PP_2} \vee S^2$ would then follow automatically from Theorem \ref{thm:finite-complexes-are-a-fork}.

The first examples of such presentations were found by Metzler in \cite{Me76} in the case $\pi_1(X)= (\Z/p)^s$ for $s \ge 3$ odd and $p \equiv 1 \mod 4$ prime. See \cite[p.~297]{HMS93} for a convenient reference.

\begin{theorem} \label{thm:Metzler}
For $s \ge 3$ odd, $p \equiv 1 \mod 4$ prime and $p \nmid  q$, consider presentations
\[ \PP_q = \langle x_1, \dots, x_s \mid x_i^p=1, [x_1^q,x_2]=1, [x_i,x_j]=1, 1 \le i < j \le s, (i,j) \ne (1,2) \rangle\]
for the group $(\Z /p)^s$. Then $X_{\PP_q} \not \simeq X_{\PP_{q'}}$ if $q(q')^{-1}$ is not a square mod $p$.
\end{theorem}

\begin{remark}
The smallest case for which this is satisfied is the case $p=5$, $s=3$, $q=1$ and $q'=2$, corresponding to the group $(\Z / 5)^3$. 
\end{remark}

To prove these complexes are not homotopy equivalent, Metzler defined the bias invariant. This is a homotopy invariant defined for all finite 2-complexes and which was later shown in \cite{Si77}, \cite{Bro79} to be a complete invariant for finite 2-complexes with finite abelian fundamental group which led to a full (simple) homotopy classification in these cases. 

We are now ready to prove Theorem~\ref{thm:no-unique-prime-decomposition-high-dimension}.

\begin{proof}[Proof of Theorem~\ref{thm:no-unique-prime-decomposition-high-dimension}]
Let $s \ge 3$ be odd, $p \equiv 1 \mod 4$ be prime, and choose $q, q' \ge 1$ such that $p \nmid q$, $p\nmid  q'$ and such that  $q(q')^{-1}$ is not a square mod $p$. Let $\PP_q$ and $\PP_q'$ be the presentations for $(\Z / p )^s$ constructed above. Then $X_{\PP_q} \not \simeq X_{\PP_q'}$ by Theorem \ref{thm:Metzler}.

Since $\Def(\PP_q)=\Def(\PP_q')$, we have $\chi(X_{\PP_q})=\chi(X_{\PP_q'})$ by Lemma~\ref{lem:euler-char} and so 
\[ X_{\PP_q} \vee S^2 \simeq_s X_{\PP_q'} \vee S^2\] 
are simply homotopy equivalent, by Theorem \ref{thm:finite-complexes-are-a-fork}.
Since $n\geq 5$ this fulfills the conditions of Corollary \ref{thm:need-for-nonsimplyconnectedcase}. Thus we see that  $M = M^n(X_{\PP_q})$ and $N = M^n(X_{\PP_q'})$ have the desired properties.
\end{proof}

\subsection{Proof of Theorem \ref{thm:no-unique-prime-decomposition-dimension-4}}

\begingroup
\renewcommand\thetheorem{\ref{thm:no-unique-prime-decomposition-dimension-4}}
\begin{theorem}
There exist 4-dimensional smooth manifolds $M$ and $N$ which are not homotopy equivalent but such that there is an orientation preserving diffeomorphism between $M\# r\cdot (S^2\times S^{2})$ and $N\# r\cdot (S^2\times S^{2})$ for some $r\geq 1$. 
\end{theorem}
\addtocounter{theorem}{-1}
\endgroup

\begin{proof}
We use the same notation as in the proof of Theorem~\ref{thm:no-unique-prime-decomposition-high-dimension}.
By Theorem~\ref{thm:thickenings-exist} there exist 5-dimensional thickenings $A$ for $X_{\PP_{q}}$ and $B$ for $X_{\PP_q'}$.  We write $M=\partial A$ and $N=\partial B$. 

As in Lemma~\ref{lem:thicken-sum-m} we see that $A\#_b (S^2\times \ol{B}^3)$ is a thickening of $X_{\PP_q}\vee S^2$ and we see that $B\#_b (S^2\times \ol{B}^3)$ is a thickening of $X_{\PP_q'}\vee S^2$. As in the proof of Theorem~\ref{thm:no-unique-prime-decomposition-high-dimension} we note that 
 $X_{\PP_q}\vee S^2$ is  simple homotopy equivalent to  $X_{\PP_q'}\vee S^2$.
Thus we obtain from Proposition~\ref{prop:unique-thickening-simple-homotopic-dim-4} that $M\# (S^2\times S^2)$ and $N\# (S^2\times S^2)$ are $s$-cobordant.  It follows from
Wall \cite[Theorem~3]{Wa64} (see also \cite[p.~149]{Sc05} and \cite[Theorem~1.1]{Qu83}) that these manifolds are diffeomorphic (via an orientation presentation diffeomorphism) after stabilisation by sufficiently many copies of $S^2 \times S^2$, i.e.
\[ M\# \us{=(r+1)\cdot (S^2\times S^2)}{\ub{ (S^2\times S^2) \# r(S^2 \times S^2)}} \,\,\cong_{\op{Diff}}\,\, N \#  \us{=(r+1)\cdot (S^2\times S^2)}{\ub{(S^2\times S^2) \# r(S^2 \times S^2) }}\]
for some $r \ge 0$.

We still need to show that $M$ and $N$ are not homotopy equivalent. 
Since we are now dealing with the case $k=4$ we cannot appeal to Proposition \ref{prop:thicken-homotopy}. 
But it is shown in \cite[Theorem~III.3]{KS84} (see also \cite[Proposition~4.3]{HK93 II}) that $M$ and $N$ are indeed not homotopy equivalent.
\end{proof}

\subsection{The D2 problem}
We will now discuss a link to the work of C.\ T.\ C.\ Wall on the structure of finite complexes as it places the examples above into the framework of a more general conjecture.

Wall asked \cite{Wa65} whether or not a D$n$ complex, i.e.\ a finite complex $X$ such that $H_{i}(X;M)=0$ and $H^i(X;M)=0$ for all $i \ge n+1$ and all finitely generated left $\Z[\pi_1(X)]$-modules $M$, is necessarily homotopy equivalent to a finite $n$-complex.
This was shown to be true in the case $n >2$ \cite[Corollary~5.1]{Wa66b} and in the case $n=1$ \cite{St68}, \cite{Sw69}. The case $n=2$ remains a major open problem and is known as Wall's D2-problem. 

\begin{d2}
	Let $X$ be a D2 complex. Is $X$ homotopy equivalent to a finite 2-complex?
\end{d2}

We say that a group $G$ has the \emph{D2-property} if the D2-problem is true for all D2 complexes $X$ with $\pi_1(X)=G$.
This is relevant to the present discussion due to the following equivalent formulation. 

Define an \emph{algebraic $2$-complex} $E=(F_*,\partial_*)$ over $\Z[ G]$ to be a chain complex consisting of an exact sequence
\[
\begin{tikzcd}
	F_2 \ar[r,"\partial_2"] & F_1 \ar[r,"\partial_1"] & F_0 \ar[r, "\partial_0"] & \Z \ar[r] & 0
\end{tikzcd}
\]
where $\Z$ is the $\Z [G]$-module with trivial $G$ action and 
where the $F_i$ are stably free $\Z [G]$-modules, i.e.\ $F_i \oplus \Z [G]^r \cong \Z[ G]^s$ for some $r,s \ge 0$.

For example, if $X$ is a finite 2-complex with a choice of polarisation $\pi_1(X) \cong G$, the chain complex $C_*(\widetilde{X})$ of the universal cover is a chain complex over $\Z [G]$ under the deck transformation action of $G$ on $\widetilde{X}$. Since the action is free, $C_i(\widetilde{X})$ is free for all $i \ge 0$ and so $C_*(\widetilde{X})$ is an algebraic $2$-complex over $\Z [G]$. We say an algebraic $2$-complex over $\Z [G]$ is \textit{geometrically realisable} if it is chain homotopy equivalent to $C_*(\widetilde{X})$ for some finite 2-complex $X$.

The following correspondence is established in \cite[Theorem 1.1]{Ni19}:

\begin{theorem}\label{thm:D2=alg}
If $G$ is a finitely presented group, then there is a one-to-one correspondence between polarised D2 complexes $X$ with $\pi_1(X) \cong G$ up to polarised homotopy and algebraic $2$-complexes over $\Z [G]$ up to chain homotopy given by $X \mapsto C_*(\widetilde{X})$.
\end{theorem}

In particular, $G$ has the D2-property if and only if every algebraic $2$-complex over $\Z[G]$ is geometrically realisable, as was already shown in \cite{Jo03} and \cite{Ma09}. 

One can thus search for further examples of finite 2-complexes $X$ and $Y$ for which $X \not \simeq Y$ and $X \vee S^2 \simeq_s Y \vee S^2$ by studying the chain homotopy types of algebraic $2$-complexes over $\Z[G]$ for $G$ having the D2-property. 

A class of groups $G$ for which it is feasible to classify algebraic $2$-complexes over $\Z [G]$ up to chain homotopy are those with \textit{$n$-periodic cohomology}, i.e.\ for which the Tate cohomology groups satisfy $\widehat{H}^i(G;\Z) = \widehat{H}^{i+n}(G;\Z)$ for all $i \in \Z$. Let $m_{\Hr}(G)$ denote the number of copies of $\Hr$ in the Wedderburn decomposition of $\R [G]$ for a finite group $G$, i.e.\ the number of one-dimensional quaternionic representations. The following is a consequence of combining Theorem \ref{thm:finite-complexes-are-a-fork} with a special case of \cite[Theorem A]{Ni20}, which is proven as an application of a recent cancellation result for projective $\Z[G]$ modules \cite{Ni18}.

\begin{theorem} \label{thm:D2}
If $G$ has 4-periodic cohomology and $m_{\Hr}(G) \ge 3$. If $G$ has the D2 property, then there exists finite 2-complexes $X$ and $Y$ with $\pi_1(X) \cong \pi_1(Y) \cong G$ for which $X \not \simeq Y$ and $X \vee S^2 \simeq_s Y \vee S^2$. 
\end{theorem}

\begin{remark}
More generally, \cite[Theorem A]{Ni20} gives non-cancellation examples for finite $n$-complexes for all even $n > 2$ without any assumption on the D2 property.	
\end{remark}

Examples of groups with $4$-periodic cohomology and $m_{\Hr}(G) \ge 3$ include the generalised quaternion groups
\[ Q_{4n} = \langle x, y \mid x^n=y^2,yxy^{-1}=x^{-1}\rangle\]
for $n \ge 6$ and the groups $Q(2^na;b,c)$ which appear in Milnor's list \cite{Mi57} for $n =3$ or $n \ge 5$, and $a, b, c$ odd coprime with $c \ne 1$ \cite[Theorem 5.10]{Ni19}.

It was shown in \cite[Theorem 7.7]{Ni19} that $Q_{28}$ has the D2-property (contrary to a previous conjecture \cite{BW05}), and so $Q_{28}$ gives an example where the hypotheses of Theorem \ref{thm:D2} hold. In fact, the examples predicted by \cref{thm:D2} were determined explicitly in \cite{MP19}:

\begin{proposition}\label{thm:q28}
Consider the following presentations for $Q_{28}$: 
\[ \PP_1= \langle x,y \mid x^7=y^2, yxy^{-1}=x^{-1} \rangle, \quad \PP_2= \langle x,y\mid x^7=y^2, y^{-1}xyx^2=x^3y^{-1}x^2y \rangle.\]
	Then $X_{\PP_1} \not \simeq X_{\PP_2}$ and $X_{\PP_1} \vee S^2 \simeq_s X_{\PP_2} \vee S^2$. 
\end{proposition}

By Corollary \ref{thm:need-for-nonsimplyconnectedcase} this shows that, for $k \ge 5$, we have that $M^k(X_{\PP_1}) \not \simeq M^k(X_{\PP_2})$ and 
\[ M^k(X_{\PP_1}) \# (S^2 \times S^{k-2}) \cong M^k(X_{\PP_2}) \# (S^2 \times S^{k-2}).\]
This gives an alternate way to prove Theorem \ref{thm:no-unique-prime-decomposition-high-dimension}, as well as giving an example whose fundamental group is non-abelian.

We conclude this section by remarking that, whilst Theorem \ref{thm:D2} gives a reasonable place to look to find further non-cancellation examples, there is currently no known method to show that such examples exist without an explicit construction. Indeed, the presentations found in Proposition \ref{thm:q28} are used in the proof of \cite[Theorem 7.7]{Ni19}.

\section{Simply-connected complexes}\label{section:proofs-simply-connected-case}
Theorem \ref{thm:simply-connected-case} follows from the next theorem together with Lemma~\ref{lem:not-cancellable}. 

\begin{theorem}
	\label{thm:simply-connected-case-new}
Let $k \ge 17$. There exist simply connected $k$-dimensional smooth manifolds $M$ and $N$
which are not homotopy equivalent but such that there is an orientation preserving diffeomorphism between $M\# (S^{5}\times S^{k-5})
$ 
and $N
\# (S^{5}\times S^{k-5})$.
\end{theorem}

\begin{remark} The bound $k \geq 17$ is an artifact of our method, and we expect similar examples 
to exist in a much lower range of dimensions.
The strict analogue of Theorem~\ref{thm:simply-connected-case-new} cannot, however, hold in dimension $4$: It follows from Donaldson's Theorem \cite[Theorem~A]{Do83}, the classification of indefinite intersection forms and Freedman's' Theorem \cite[Theorem~1.5]{Fr82} that  any two 4-dimensional  simply connected \emph{smooth}  manifolds that become diffeomorphic after the connected sum with $r\cdot (S^2\times S^2)$ where $r\geq 1$, are already homeomorphic.
\end{remark}

The key idea is once again to use Theorem~\ref{thm:need-for-nonsimplyconnectedcase}.
But this time we will use  mapping cones to produce useful CW-complexes. We introduce the following notation.

\begin{notation}
Let $\alpha\colon S^{m-1}\to S^n$ be a map. We denote its mapping cone by $C_\alpha$.
Note that $C_\alpha$ has a CW-structure with three cells, one in dimension 0, one in dimension $n$ and one in dimension $m$.
\end{notation}

The following theorem is a practical machine for constructing interesting CW-complexes, see \cite[Theorem 3.1 \& Corollary 3.3]{Hi67}. Note that \cite{Hi67,Mo73} contain many other examples of CW-complexes exhibiting similar phenomena.

\begin{theorem}\label{thm:hilton}
Let $m,n\in \N_{\geq 3}$ and let $[\alpha],[\beta]\in \pi_{m-1}(S^n)$ be elements of finite order. If $[\alpha]$ is in the image of the suspension homomorphism $\pi_{m-2}(S^{n-1})\to \pi_{m-1}(S^n)$, then the following two statements hold:
\bnm[font=\normalfont]
\item $C_\alpha\simeq C_\beta$ if and only if $[\beta]=\pm [\alpha]\in \pi_{m-1}(S^n)$.
\item If $[\alpha]$ and $[\beta]$ generate the same subgroup of $\pi_{m-1}(S^n)$, then $C_\alpha\vee S^m\simeq C_\beta \vee S^m$.
\enm
\end{theorem}
\begin{proof}
We prove (1) first and may assume that $m\geq n+2$, else the statements become easy (if $m-1 =n$) or trivial (if $m-1 < n$). The ``if'' part is obvious. To see the ``only if'' we consider a homotopy equivalence $f\colon C_\alpha \simeq C_\beta$. Such a homotopy equivalence induces an isomorphism on $\pi_n$, so that we find that there is a homotopy commutative diagram
\[ \begin{tikzcd}
	S^n \ar[r] \ar[d,"\pm1"] & C_\alpha \ar[d,"f"] \\
	S^n \ar[r] & C_\beta   
\end{tikzcd}\]
in which the horizontal maps are the canonical maps. We denote by $F_\alpha$ and $F_\beta$ the homotopy fibres of these horizontal maps. Since the composites $S^{m-1} \to S^n \to C$ ($C=C_\alpha,C_\beta$) are null homotopic, we obtain a canonical homotopy commutative diagram
\[\begin{tikzcd}
	S^{m-1} \ar[r] \ar[d, dashed,"\pm 1"] & F_\alpha \ar[r] \ar[d,"\bar{f}"] & S^n \ar[r] \ar[d,"\pm 1"] & C_\alpha \ar[d,"f"] \\
	S^{m-1} \ar[r] & F_\beta \ar[r] & S^n \ar[r] & C_\beta
\end{tikzcd}\]
where $\bar{f}$ is the induced homotopy equivalence of homotopy fibres. The relative Hurewicz theorem for the maps $S^n \to C_\alpha$ and $S^n \to C_\beta$, together with $m\geq n+2$, implies that the two maps $S^{m-1} \to F_\alpha$ and $S^{m-1} \to F_\beta$ induce isomorphisms on $\pi_{m-1}$ so one obtains a dashed arrow making the diagrams commute up to homotopy.
We deduce that there is a homotopy commutative diagram
\[\begin{tikzcd}
	S^{m-1} \ar[r,"\alpha"] \ar[d,"\pm1"] & S^n \ar[d,"\pm1"] \\
	S^{m-1} \ar[r,"\beta"] & S^n
\end{tikzcd}\]
Now we use that $\alpha$ is a suspension, so that post composition of $\alpha$ with a degree $-1$ map is just $-\alpha$ in $\pi_{m-1}(S^n)$.

To see (2) we consider the space $C_{\alpha,\beta} = C_\alpha \cup_{\beta} D^m$. Here, we view $\beta$ as the composite $S^{m-1} \to S^n \to C_\alpha$. We note that $C_{\alpha,\beta} \cong C_{\beta,\alpha}$. By assumption $\beta$ is contained in the subgroup generated by $\alpha$. This implies the composite $S^{m-1} \to S^n \to C_\alpha$ is null homotopic so that $C_{\alpha,\beta} \simeq C_\alpha \vee S^m$. We thus obtain
\[ C_\alpha \vee S^m \simeq C_{\alpha,\beta} \simeq C_{\beta,\alpha} \simeq C_\beta \vee S^m \]
where the very last equivalence follows from the assumption that $\alpha$ is also contained in the subgroup generated by $\beta$.
\end{proof}

\begin{proof}[Proof of Theorem~\ref{thm:simply-connected-case-new}]
It is well known that the $\pi_6(S^3) \cong \Z/12$
and that suspension homomorphism $\pi_6(S^3) \to \pi_7(S^4)$ is injective with image $t(\pi_7(S^4))$;
see \cite[Proposition 5.6, Lemma 13.5]{To62}.
Let $\mu$ be a generator of $t(\pi_7(S^4))$.
We then consider the elements $\alpha = \mu$ and $\beta = 5 \cdot \mu$ in $t(\pi_7(S^4))$.

%
It follows from Theorem~\ref{thm:hilton} that 
$C_\alpha \not\simeq C_\beta$ and that $C_{\alpha} \vee S^8 \simeq C_{\beta}\vee S^8$. 
Since these CW-complexes are simply connected, we have in fact 
 $C_{\alpha} \vee S^8 \simeq_s C_{\beta}\vee S^8$.
The theorem is now an immediate consequence of Corollary~\ref{thm:need-for-nonsimplyconnectedcase}
applied to the 8-dimensional CW-complexes $X=C_\alpha$, $Y=C_\beta$, $n=4$ and the given $k \geq 17$. 
\end{proof}


\section{Prime manifolds}\label{section:prime-manifolds}
Let $\M{\Cat,\sc}{n}$ denote the submonoid of $\M{\Cat}{n}$ of simply connected manifolds. The question we want to ask in the present section is whether there exist prime manifolds in higher dimensions at all. While we do not know the answer, we will show that the Wu-manifold is prime among simply connected $5$-folds.

As a warm-up recall that on the one hand the manifolds $\cp^n,\ol{\cp^n}$ and  $S^n\times S^n$ are all irreducible by \cref{cor:examples-topologically-irreducible}, except possibly for $n=2$ and $\Cat = \Diff$ or $\PL$.
On the other, as mentioned before, $(S^2\times S^2)\#\cp^2$ and $\cp^2\# \ol{\cp}^2\# \cp^2$ are well-known to be diffeomorphic, see e.g.\ \cite[p.~151]{GS99} for details. These two observations imply immediately that none of $S^2\times S^2, \cp^2$ or $\overline{\cp^2}$ are prime in $\M{\Cat,\sc}{4}$ or $\M{\Cat}{4}$. In higher dimension we recorded similar behaviour for $S^{2k} \times S^{2k}$ in the first lines of the proof of Theorem \ref{thm:wall-highly-connected}. 
 
\begin{corollary}\label{prop:sn-times-sn-not-prime}
	Let  $n\in \N_{\geq 2}$ be even. Then $S^n\times S^n$ is not prime in $\M{\Cat,\sc}{2n}$ or $\M{\Cat}{2n}$.
\end{corollary}

By contrast, for some odd $n$ Theorem \ref{thm:wall-highly-connected} also implies that $S^n \times S^n$ is prime in $\M{\Diff,\high}{2n}$. We do not know whether this extends to simply-connected manifolds, i.e.\ we do no know whether for those odd   $S^n \times S^n$ is prime in $\M{\Diff,\sc}{2n}$.

\begin{proposition}\label{prop:4dim}
	The monoid $\M{\Cat,\sc}{4}$ has no prime elements. In particular, no simply-connected manifold is prime in $\M{\Cat}{4}$.
\end{proposition}

\begin{proof}
	By Freedman's classification, two simply connected, topological 4-manifolds are homeomorphic if and only if they have isomorphic intersection forms and the same Kirby-Siebenmann invariant. Hence for any such manifolds $M$ there exist $m,m',n,n'\in\N, \epsilon\in\{0,1\}$ such that $M\#m\cp^2\#m'\ol{\cp^2}\# \epsilon*\cp^2$ and $n\cp^2\#n'\ol{\cp^2}$ are homeomorphic. Similarly, if $M$ is assumed smooth (or piecewise linear), it follows from the above and work of Wall \cite[Theorem~3]{Wa64} (see also \cite[p.~149]{Sc05} and \cite[Theorem~1.1]{Qu83}) that there exist $m,m',n,n'$ and $k\in \N$ such that 
	$M\#m\cp^2\#m'\ol{\cp^2}\#k (S^2\times S^2)$ and $n\cp^2\#n'\ol{\cp^2}\#k (S^2\times S^2)$ are diffeomorphic. Using the fact that  $(S^2\times S^2)\#\cp^2$ and $\cp^2\# \ol{\cp}^2\# \cp^2$ are diffeomorphic we can arrange, at the cost of increasing $m,m',n,n'$ that $k=0$. 
	 In either case, if $M$ is prime it follows that $M$ is either $\cp^2$ or $\overline{\cp^2}$, since the latter manifolds are irreducible. But we observed above that they are not prime.
\end{proof}

Turning to dimension $5$, recall that the Wu manifold $\op{SU}(3)/\op{SO}(3)$ is a simply connected, non-spin $5$-manifold with $H_2(W;\mathbb Z) = \Z/2$. 

\begin{proposition}\label{Wu}
	The Wu-manifold $W=\op{SU}(3)/\op{SO}(3)$ is prime in $\M{\Cat,\sc}{5}$.
\end{proposition}
We will use Barden's classification of smooth simply connected 5-manifolds \cite{Ba65,Cr11}. There are two invariants which are important to us:
\begin{enumerate}
\item $H_2(M;\Z)$ with its torsion subgroup $TH_2(M;\Z)$. The group $TH_2(M;\Z)$ is always isomorphic to $A\oplus A \oplus C$ where $C$ is either trivial or cyclic of order 2 and $A$ is some finite abelian group.
\item The height $h(M) \in \mathbb{N}_0 \cup \{\infty\}$: If $M$ is spin, one sets $h(M) = 0$. If $M$ is non-spin, $w_2 \colon H_2(M;\Z) \to \Z/2$ is a surjection. It is an algebraic fact that for any surjection $w\colon H \to \Z/2$ where $H$ is a finitely generated abelian group, there exists an isomorphism $H \cong H' \oplus \Z/2^\ell$ such that $w$ corresponds to the composite $H' \oplus \Z/2^\ell \to \Z/2^\ell \to \Z/2$. Here, $\ell$ is allowed to be $\infty$, where we (ab)use the notation that $\Z/2^\infty = \Z$. The number $\ell$ is then defined to be the height $h(M)$ of $M$. Note that here we follow the wording of \cite{Ba65}, an equivalent definition of the height is given in \cite{Cr11}. 
\end{enumerate}

Barden's classification says that the map $\M{\Diff,\sc}{5} \to \mathrm{Ab} \times (\mathbb{N}_0 \cup \{\infty\})$ sending a manifold $M$ to the pair $(H_2(M;\Z),h(M))$ is injective, and that the following two statements are equivalent:
\bnm
\item a pair $(B,k)$ lies in the image,
\item the torsion subgroup $TB$ is of the form $A \oplus A \oplus \Z/2$ if  $k=1$ and it is of the form $A\oplus A$ otherwise.
\enm
Moreover, the above map restricts to a bijection between spin manifolds and the pairs $(B,0)$ where $TB \cong A \oplus A$.
\begin{lemma}\label{lem:barden}
\begin{enumerate}[font=\normalfont]
\item $h(M\# N) = \begin{cases} h(M)&\text{ if }h(N)=0,\\ h(N) & \text{ if } h(M) = 0, \\ \min(h(M),h(N)) & \text{ if } h(M) \neq 0 \neq h(N) \end{cases}$
\item $h(M\# N) = 1$ if and only if $h(M) = 1$ or $h(N) = 1$,
\item The Wu manifold $W$ divides $M$ if and only if $h(M) = 1$,
\end{enumerate}
\end{lemma}
\begin{proof}
To see (1), we observe that $M\#N$ is spin if and only if both $M$ and $N$ are spin. Furthermore, it is clear from the above definition of the height that if $M$ is spin, then $h(M\#N) = h(N)$. To see the case where both $M$ and $N$ are not spin, it suffices to argue that if $\ell \leq k$, and we consider the map $\Z/2^k \oplus \Z/2^\ell \to \Z/2$ which is the sum of the canonical projections, then there is an automorphism of $\Z/2^k \oplus \Z/2^\ell$ such that this map corresponds to the map $\Z/2^k \oplus \Z/^\ell \to \Z/2^\ell \to \Z/2$: The automorphism is given by sending $(1,0)$ to $(1,1)$ and $(0,1)$ to $(0,1)$. Statement (2) is then an immediate consequence of (1). To see (3) we first assume that $W$ divides $M$, i.e.\ that $M$ is diffeomorphic to $W \# L$. We find that $h(W\# L) = 1$ by (1). Conversely, suppose that $h(M) = 1$. By Barden's classification, we know that the \emph{torsion} subgroup of $H_2(M;\Z)$ is of the form $A \oplus A \oplus \Z/2$, in particular $H_2(M;\Z) \cong \Z^n \oplus A \oplus A \oplus \Z/2$ for some $n \geq 0$. Again, by the classification, there exists a spin manifold $L$ with $H_2(L;\Z) \cong \Z^n \oplus A \oplus A$. We find that $W \# L$ and $M$ have isomorphic $H_2(-;\Z)$ and both height $1$, so they are diffeomorphic, and thus $W$ divides $M$.
\end{proof}

\begin{proof}[Proof of Proposition \ref{Wu}]
	First we consider $\M{\PL,\sc}{5}=\M{\Diff,\sc}{5}$. If $W$ divides $M\# N$, then $h(M \#N )=1$, we may then without loss of generality assume that $h(M) =1$, so that $W$ divides $M$.
	
	Every simply connected topological 5-manifold admits a smooth structure, since the Kirby-Siebenmann invariant lies in $H^4(M;\Z/2) \cong H_1(M;\Z/2) = 0$. As the invariants in Barden's classification are homotopy invariants, it follows that $\M{\Top,\sc}{5}=\M{\Diff,\sc}{5}$. In particular, the Wu manifold is also prime in $\M{\Top,\sc}{5}$.
\end{proof}

\section{Questions and problems} \label{section:questions}
We conclude this paper with a few questions and challenges.

\begin{question}
\mbox{}
\bnm[font=\normalfont]
\item Let $n\geq 4$. Does there exist a  non-trivial cancellable element in any of the monoids $\M{\Top}{n}$,
$\M{\PL}{n}$,
$\M{\Diff}{n}$?
\item Let $n\geq 6$.  Does there exist a   non-trivial cancellable element in any of the monoids $\M{\Top,\sc}{n}$,
$\M{\PL,\sc}{n}$,
$\M{\Diff,\sc}{n}$?
\enm
\end{question}

\begin{question}
\mbox{}
\bnm[font=\normalfont]
\item Let $n\geq 4$.  Does there exist a prime element in any of the monoids $\M{\Top}{n}$,
$\M{\PL}{n}$,
$\M{\Diff}{n}$?
\item Let $n\geq 6$.  Does there exist a prime element in any of the monoids $\M{\Top,\sc}{n}$,
$\M{\PL,\sc}{n}$?
\enm
\end{question}

In light of  Theorem~\ref{thm:simply-connected-case} and the remark on page~\pageref{rem:17} we also raise the following related question.

\begin{question}
For which $n\in 5,\dots,16$ is  $\M{\Cat,\sc}{n}$  a unique factorisation monoid?
\end{question}

The following question arises naturally from Proposition~\ref{Wu}.

\begin{question}
	Is the Wu manifold prime in $\M{\Top}{5}$ or $\M{\Diff}{5}$?
\end{question}

Throughout the paper we worked mostly with simply connected and highly connected manifolds. 
It is reasonable to ask what is happening at the end of the spectrum, namely when we restrict ourselves to aspherical manifolds. 

\begin{question} \label{qn}
In any of the three categories $\Top$, $\PL$ and $\Diff$, is the monoid generated by aspherical  manifolds  a unique decomposition factorization monoid for $n\geq 4$?
\end{question}

As a partial answer to Question \ref{qn},
we would like to thank the referee for pointing out that
in the topological category 
unique decomposition factorization is implied by the Borel conjecture.
To see this, 
we first note that the fundamental group of an aspherical manifold can not be a non-trivial free product of groups.
\begin{lemma}
	\label{lem:indec}
	Let $M$ be aspherical closed $n$-manifold and let $\pi_1(M)\cong G*H$. Then either $G$ or $H$ is trivial.
\end{lemma}
\begin{proof}
As $M$ is aspherical the assumption that $\pi_1(M) \cong G * H$ 
implies that $M$ is homotopy equivalent to $B(G*H)\simeq BG\vee BH$.
Note that for all $k\in \N$ we have 
$$\widetilde H_k(BG\vee BH;\Z/2)\cong \widetilde H_k(BG;\Z/2)\oplus \widetilde H_k(BH;\Z/2).$$
%
Hence $\Z/2\cong H_n(M;\Z/2)\cong H_n(BG\vee BH;\Z/2)\cong H_n(BG;\Z/2)\oplus H_n(BH;\Z/2)$. 
Now we suppose without loss of generality that $H_n(BH;\Z/2)$ is trivial. Consider the cover $\widehat{M}$ of $M$ corresponding to the canonical map $\pi_1(M)\cong G*H\to H$. Then $\widehat{M}$ is homotopy equivalent to $\bigvee_{i=1}^{|H|} BG$ and hence $H_n(\widehat{M};\Z/2)\cong \bigoplus_{i=1}^{|H|} \Z/2$. As $\widehat{M}$ is a manifold, this implies $|H|=1$ and thus $H$ is trivial. 
\end{proof}

Recall that by the Grushko--Neumann theorem \cite[Corollay~IV.1.9]{LS77}
together with the Kurosh isomorphism theorem \cite[Isomorphiesatz]{Ku34}\cite[p.~27]{Ku60} we obtain the following lemma.
\begin{lemma}
	\label{lem:grushko}
	Every non-trivial, finitely generated group $G$ can be decomposed as a free product
	\[ G\cong A_1*\ldots *A_r*F_k, \]
	where $F_k$ is a free group of rank $k$, each of the groups $A_i$ is non-trivial, freely indecomposable and not infinite cyclic; moreover, for a given $G$, the numbers $r$ and $k$ are uniquely determined and the groups $A_1,\ldots, A_r$ are unique up to reordering and conjugation in $G$. That is, if $G \cong B_1*\ldots B_s*F_l$ is another such decomposition then $r=s$, $k=l$, and there exists a permutation $\sigma\in S_r$ such that for each $i=1,\ldots,r$ the subgroups $A_i$ and $B_{\sigma(i)}$ are conjugate in $G$. 
\end{lemma}

\begin{proposition}\label{prop:borel-conjecture-implies-unique-dec}
	Assume that the Borel conjecture is true. Then in $\Top$ the monoid generated by aspherical manifolds is a unique decomposition factorization monoid for $n\geq 4$.
\end{proposition}
\begin{proof}
	Let $N:=N_1\#\ldots  \# N_k$ and $M:=M_1\#\ldots \# M_{l}$ with $N_i$ and $M_j$ aspherical for all $i$ and $j$. Crushing all connecting spheres to points, we obtain projection maps $p_N\colon N\to N_1\vee\ldots\vee N_k$ and $p_M\colon M\to M_1\vee\ldots\vee M_l$. 
	Suppose there is an orientation preserving homeomorphism $f\colon N \xrightarrow{\cong} M$. The maps $p_N$ and $p_M$ induce isomorphisms on the fundamental groups by \cref{lem:invariants-additive}.\eqref{it:fundamentalgroup}. We consider the isomorphism $\varphi:= p_{M*}\circ f_*\circ p_{N*}^{-1}\colon 
	\pi_1(N_1)*\dots*\pi_1(N_k)\to \pi_1(M_1)*\dots*\pi_1(M_l)$.  Note that $\pi_1(N_j)$ is never infinite cyclic since the dimension of $N_j$ is larger than one. By \cref{lem:indec} and \cref{lem:grushko}, $k=l$ and for each $j$ there is an $i$ with $\pi_1(N_j)\cong \pi_1(M_i)$. Moreover, since $\varphi(\pi_1(N_j))$ is conjugate to $\pi_1(M_i)$ in $\pi_1(M)$,
	we see that the map $\pi_1(N_j)\to \pi_1(N_1)*\dots* N_k\xrightarrow{\varphi} \pi_1(M_1)* \dots * \pi_1(M_k)\to \pi_1(M_i)$ is an isomorphism. Hence it induces a homotopy equivalence $N_j\to M_i$.  The fundamental class of $N$ is mapped to the fundamental classes $([N_j])_j\in\bigoplus_{r=1}^kH_n(N_r;\Z)\cong H_n(N_1\vee\ldots\vee N_k;\Z)$. Hence the homotopy equivalence $N_j\to M_i$ is orientation preserving. Assuming the Borel conjecture, it follows that $N_j$ is orientation preserving homeomorphic to $M_i$. Thus the decomposition is unique. 
\end{proof}


\begin{remark}
Finally, in the smooth and $PL$ categories in dimensions $n \geq 5$,
we also thank the referee  for suggesting that the existence of
exotic tori might lead to the failure of unique factorisation in the
monoid generated by aspherical manifolds:
We think that this is an attractive approach to attacking Question \ref{qn}.
\end{remark}


\begin{thebibliography}{10}
\bibitem[Ba65]{Ba65}
D. Barden. {\em Simply connected five-manifolds},  Ann. of Math. 82 (1965),  365--385.
\bibitem[Bay80]{Bay80}
E. Bayer. {\em Factorisation is not unique for higher dimensional knots}, Comment. Math. Helv. 55, 583-592 (1980).
\bibitem[BHHM20]{BHHM}
M. Behrens, M. Hill, M. Hopkins and M. Mahowald. {\em Detecting exotic spheres in low dimensions using coker J}, J. Lond. Math. Soc. (2) 101 (2020), no. 3, 1173--1218.
\bibitem[BW05]{BW05} F. R. Beyl and N. Waller. {\em A stably free nonfree module and its relevance for homotopy classification, case $Q_{28}$}, Algebr. Geom. Topol. 5 (2005), 899--910.
\bibitem[Bre93]{Bre93}
G. Bredon. {\em Topology and geometry}, Graduate Texts in Mathematics 139. Springer-Verlag (1993)
\bibitem[Bro79]{Bro79} W. H. Browning. {\em Homotopy types of certain finite CW-complexes with finite fundamental group}, PhD Thesis, Cornell University (1979).
\bibitem[Ca74]{Ca74} S. E. Cappell, {\em On connected sums of manifolds}, Topology 13 (1974), 395--400.
\bibitem[Ca76]{Ca76} S. E. Cappell, {\em A spitting theorem for manifolds}, Inventiones Math.\ 33 (1976), 69--170.
\bibitem[CS71]{CS71} 
S. E. Cappell and J. L. Shaneson. {\em On four-dimensional surgery and applications}, Comment. Math. Helv. 46 (1971), 500--528.
\bibitem[Ce68]{Ce68}
J. Cerf. {\em Sur les diff\'eomorphismes de la sph\`ere de dimension trois $(\Gamma_4=0)$}, Lecture Notes in Mathematics 53, Springer Verlag (1968)
\bibitem[CK70]{CK70}
J. Cheeger and J. M. Kister. {\em Counting topological manifolds}, Topology 9 (1970), 149--151.
\bibitem[Cr11]{Cr11}
D. Crowley. {\em 5-manifolds: 1-connected}, Bulletin of the Manifold Atlas (2011), 49-55.
\bibitem[Do83]{Do83}
S. K. Donaldson. {\em An application of gauge theory to four-dimensional topology}, J. Diff. Geom. 18
(1983), 279--315.
\bibitem[Ed84]{Ed84}
R. D. Edwards. {\em The solution of the 4-di\-men\-si\-onal Annulus conjecture $($after Frank Quinn$)$},
in ``Four-manifold Theory'', Gordon and Kirby ed., Contemporary Math. 35
(1984), 211--264.
\bibitem[FNOP19]{FNOP19}
S. Friedl,  M. Nagel, P. Orson and M. Powell.
{\em A survey of the foundations of $4$-manifold theory}, Preprint (2019), arXiv:1910.07372.
\bibitem[Fr82]{Fr82}
M. Freedman. {\em The topology of four-dimensional manifolds}, J. Diff. Geom. 17 (1982), 357--453.
\bibitem[Fr38]{Fr38}
H. Freudenthal. {\em \"{U}ber die Klassen der Sph\"{a}renabbildungen I. Gro{\ss}e Dimensionen},  Compos. Math. 5 (1938), 299--314
\bibitem[GS99]{GS99}
R. Gompf and A. Stipsicz. {\em 4-manifolds and Kirby calculus}, Graduate Studies in Mathematics 20,
AMS (1999).
\bibitem[Ha01]{Ha01}
A. Hatcher. {\em Algebraic Topology}, Cambridge University Press (2001)
\bibitem[Hi55]{Hi55}
P. J. Hilton. {\em On the homotopy groups of the union of spheres}, J. London Math. Soc., Second Series, 30 (2) (1955), 154--172.
\bibitem[Hi65]{Hi65}
P. J. Hilton. {\em Homotopy Theory and Duality}, Gordon and Breach, New York, (1965).
\bibitem[Hi67]{Hi67}
P. J. Hilton. {\em On the Grothendieck group of compact polyhedra}, Fund. Math. 61 (1967), 199--214.
\bibitem[HK88]{HK88}
I. Hambleton and M. Kreck. {\em On the classification of topological 4-manifolds with finite fundamental group}, Math. Annalen. 280 (1988), 85--104.
\bibitem[HK93a]{HK93 I} 
I. Hambleton and M. Kreck. {\em Cancellation of lattices and finite two-complexes}, J. reine angew. Math 442 (1993), 91--109.
\bibitem[HK93b]{HK93 II}
I. Hambleton and M. Kreck. {\em Cancellation results for 2-complexes and 4-manifolds and some applications}, Two-dimensional homotopy and combinatorial group theory, 281-308, London Math. Soc. Lecture Note Ser., 197, Cambridge Univ. Press, Cambridge, 1993.
\bibitem[He76]{He76}
J. Hempel. {\em 3-manifolds}, 
Annals of Mathematics Studies 86. Princeton, New Jersey: Princeton University Press and University of Tokyo Press. XII (1976).
\bibitem[HHR16]{HHR}
M. Hill, M. Hopkins and D. Ravenel. {\em On the nonexistence of elements of Kervaire invariant one}, Ann. of Math. (2) 184 (2016), 1 - 262.
\bibitem[HMS93]{HMS93}
C. Hog-Angeloni, W. Metzler and A. J. Sieradski. {\em Two-Dimensional Homotopy and Combinatorial Group Theory,} London Math. Soc. Lecture Note Ser. 197, Cambridge University Press (1993).
\bibitem[Jo03]{Jo03} 
F. E. A. Johnson. {\em Stable Modules and the D(2)-Problem}, LMS Lecture Notes Series 301 (2003).
\bibitem[JR78]{Jones} 
J. Jones and E. Rees. {\em Kervaire's invariant for framed manifolds}, Proc. Symp. Pure Math. 32 (1978), 141--147.
\bibitem[Ke79]{Ke79}
C. Kearton. {\em Factorisation is not unique for 3-knots},
Indiana Univ. Math. J. 28, 451-452 (1979). 
\bibitem[Ke80]{Ke80}
C. Kearton. {\em 
The factorisation of knots},
Low-dimensional topology, Proc. Conf., Bangor/Engl. 1979, Vol. 1, Lond. Math. Soc. Lect. Note Ser. 48, 71-80 (1982). 
\bibitem[KM63]{KM63}
M. Kervaire and J. Milnor. {\em Groups of Homotopy Spheres: I},
Annals of Mathematics 77 (1963), 504--537.
\bibitem[Ki69]{Ki69}
R. Kirby. {\em Stable homeomorphisms and the annulus conjecture}, Annals of
Math. 89 (1969), 575--582.
\bibitem[Kn29]{Kn29}
H. Kneser. {\em Geschlossene Fl\"achen in dreidimensionalen Mannigfaltigkeiten}, Jber.
Deutsch. Math.-Verein. 38 (1929), 248--260.
\bibitem[Ko93]{Kosinski}
A. Kosinski. {\em Differential Manifolds}, Dover Publications, 1993.
\bibitem[KLT95]{KLT95}
M. Kreck, W. L\"uck and P. Teichner,
{\em Counterexamples to the Kneser conjecture in dimension four},
 Comment.\ Math.\ Helvetici 70 (1995), 423--433.
\bibitem[KS84]{KS84}
M. Kreck and J. A. Schafer. {\em Classification and stable classification of manifolds: some examples}, Comment. Math. Helv. 59 (1984), 12--38.
\bibitem[Kr99]{Kr99} M. Kreck. {\em Surgery and duality}, Ann. of Math. (2) 149 (1999), 707--754.
\bibitem[Ku34]{Ku34}
A. G. Kurosh. {\em Die Untergruppen der freien Produkte von beliebigen Gruppen}, Math. 
Annalen 109 (1934), 647--660. 
\bibitem[Ku60]{Ku60}
A. G. Kurosh. {\em  The theory of groups}, Volume 2,  Chelsea Publishing Co. (1960).
\bibitem[LS77]{LS77}
R. Lyndon and P. Schupp. {\em Combinatorial group theory}, Springer Verlag (1977).
\bibitem[MP19]{MP19} 
W. Mannan and T. Popiel. {\em An exotic presentation of $Q_{28}$}, arXiv:1901.10786 (2019).
\bibitem[Ma09]{Ma09}
W. H. Mannan. {\em Realizing algebraic 2-complexes by CW-complexes}, Math. Proc. Cam. Phil. Soc. 146 (3) (2009), 671--673.
\bibitem[Me76]{Me76} W. Metzler. {\em \"{U}ber den Homotopietyp zweidimensionaler CW-Komplexe und Elementartransformationen bei Darstellungen von Gruppen durch Erzeugende und definierende Relationen}, J. reine u. angew. Math. 285 (1976), 7--23.
\bibitem[Mi62]{Mi62}
J. Milnor. {\em A unique factorisation theorem for 3-manifolds},
Amer. J. Math. 84 (1962), 1-7.
\bibitem[Mi57]{Mi57} 
J. Milnor. {\em Groups which act on $S^n$ without fixed points}, Amer. J. Math. 79 (3) (1957), 623--630.
\bibitem[Mo52]{Mo52}
E. Moise. {\em Affine structures in~$3$-manifolds V. The triangulation theorem and Hauptvermutung}, Ann. of Math. 56 (1952), 96--114.
\bibitem[Mo77]{Mo77}
E. Moise. {\em  Geometric Topology in Dimensions 2 and 3}, Graduate Texts in Mathematics, vol. 47, Springer, New York-Heidelberg, 1977.
\bibitem[Mo73]{Mo73}
E. A. Molnar. {\em Relation between cancellation and localization for complexes with two cells}, J. Pure Appl. Algebra 3 (1973), 141--158.
\bibitem[Ni18]{Ni18} 
J. Nicholson. {\em A cancellation theorem for modules over integral group rings}, Math. Proc. Cambridge Philos. Soc., to appear, arXiv:1807.00307 (2018).
\bibitem[Ni20]{Ni20}
J. Nicholson. {\em Cancellation for $(G,n)$-complexes and the Swan finiteness obstruction}, arXiv:2005.01664 (2020).
\bibitem[Ni19]{Ni19}
J. Nicholson. {\em On CW-complexes over groups with periodic cohomology}, arXiv:1905.12018 (2019).
\bibitem[Qu82]{Qu82}
F. Quinn. {\em Ends of maps. III: Dimensions 4 and 5},
J. Differ. Geom. 17 (1982), 503--521.
\bibitem[Qu83]{Qu83}
F. Quinn. {\em The stable topology of 4-manifolds}, Topology Appl. 15 (1983), 71--77.
\bibitem[Rad26]{Rad26}
T. Rad\'o. {\em \"Uber den Begriff der Riemannschen Fl\"ache}, Acta Szeged~2 (1926), 101--121.
\bibitem[Ram16]{Ram16}
K. Ramesh. {\em Inertia groups and smooth structures of $(n-1)$-connected $2n$-manifolds}, Osaka J. Math. 53 (2016), 303--319.
\bibitem[RS72]{RS72}
C. P. Rourke and B. J. Sanderson. {\em
Introduction to piecewise-linear topology}, 
Ergebnisse der Mathematik und ihrer Grenzgebiete.  Springer-Verlag (1972). 
\bibitem[Sch49]{Sch49}
H. Schubert. {\em  Die eindeutige Zerlegbarkeit eines Knotens in Primknoten},
S.-B. Heidelberger Akad. Wiss. Math.-Nat. Kl. 1949, (1949). no. 3, 57--104. 
\bibitem[Sc05]{Sc05}
A. Scorpan. {\em The wild world of 4-manifolds}, Providence, RI: American Mathematical Society (2005).
\bibitem[Si77]{Si77} A. J. Sieradski. {\em A semigroup of simple homotopy types}, Math. Zeit. 153 (1977), 135--148.
\bibitem[Sm62]{Sm62}
S. Smale. {\em On the structure of manifolds}, Amer. J. Math. 84 (1962), 387--399. 
\bibitem[St63]{Stong} 
R. E. Stong. {\em Determination of $H^{\ast} (BO\langle k\rangle,\mathbb Z/2)$ and $H^{\ast}(BU\langle k\rangle, \mathbb Z/2)$}, Trans. Amer. Math. Soc. 107, no. 3 (1963), 526-544.
\bibitem[St68]{St68} J. R. Stallings. {\em On torsion free groups with infinitely many ends}, Ann. of Math. 88 (1968), 312--334.
\bibitem[Sw69]{Sw69} R. G. Swan. {\em Groups of cohomological dimension one}, J. of Algebra. 12 (1969), 585--610.
\bibitem[To62]{To62} 
H. Toda, {\em Composition methods in homotopy groups of spheres}, Princeton University Press, 1962.
\bibitem[Wa62]{Wa62}
C. T. C. Wall. {\em Classification of $(n-1)$-connected 2n-manifolds}, Ann. of
Math. 75 (1962), 163--189.
\bibitem[Wa64]{Wa64}
C. T. C. Wall. {\em On simply-connected 4-manifolds}, J. London Math. Soc. 39 (1964), 141--149.
\bibitem[Wa65]{Wa65} C. T. C. Wall. {\em Finiteness conditions for CW Complexes}, Ann. of Math. 81 (1965), 56--69.
\bibitem[Wa66a]{Wa66} C. T. C. Wall. {\em Classification problems in Differential Topology - IV Thickenings}, Topology 5 (1966), 73--94.
\bibitem[Wa66b]{Wa66b} C. T. C. Wall. {\em Finiteness conditions for CW complexes}, II. Proc. Roy. Soc. Ser. A, 295 (1966), 129--139.
\bibitem[Wa67]{Wall67}
C. T. C. Wall. {\em Classification problems in differential topology. VI. Classification of $(s-1)$-connected $(2s+1)$-manifolds}. 
Topology 6 (1967), 273-296.
\bibitem[Wa16]{Wa16}
C. T. C. Wall. {\em 
Differential topology}, 
Cambridge Studies in Advanced Mathematics 156.  Cambridge University Press (2016).
\bibitem[WX17]{WX}
G. Wang and Z. Xu, \emph{The triviality of the 61-stem in the stable homotopy groups of spheres}.
Ann. of Math. (2) 186 (2017), no. 2, 501 - 580. 
\bibitem[Whi78]{Whi78}
G. W. Whitehead. {\em Elements of Homotopy Theory},  Graduate Texts in Mathematics 61, Springer-Verlag, New York, (1978).
\bibitem[Whi61]{Wh61}
J. H. C. Whitehead. {\em Manifolds with transverse fields in euclidean space}, Ann. of Math. (2) 73
(1961), 154-212.
\bibitem[Wi74]{Wilkens}
D. L. Wilkens. {\em On the inertia group of certain manifolds}, J. London Math. Soc. {9} (1974/75), 537--548.
\end{thebibliography}
\end{document}